\documentclass{amsproc}

\usepackage[margin=1in]{geometry}
\usepackage{amsmath,amsthm,amssymb,amsrefs,bbm,color,enumitem,esint,esvect,float,graphicx,mathrsfs}
\usepackage{euscript,latexsym}
\usepackage[OT2,T1]{fontenc}
\DeclareSymbolFont{cyrletters}{OT2}{wncyr}{m}{n}
\DeclareMathSymbol{\Sha}{\mathalpha}{cyrletters}{"58}
\usepackage{accents}
\usepackage{graphicx}
\usepackage{relsize}
\graphicspath{ {./Documents/} }

\newcommand{\N}{\mathbb{N}}
\newcommand{\Z}{\mathbb{Z}}
\newcommand{\R}{\mathbb{R}}

\newcommand{\D}{\mathscr{D}}
\newcommand{\Ss}{\mathcal{S}}

\newcommand{\A}{\mathcal{A}}
\newcommand{\C}{\mathcal{C}}

\newcommand{\dyad}{\operatorname{dist}_{\D}}

\DeclareFontFamily{U}{wncy}{}
\DeclareFontShape{U}{wncy}{m}{n}{<->wncyr10}{}
\DeclareSymbolFont{mcy}{U}{wncy}{m}{n}
\DeclareMathSymbol{\Sh}{\mathord}{mcy}{"58} 

\newcommand{\Lp}{L^p}
\newcommand{\Ltwo}{L^2}
\newcommand{\Lone}{L^1}
\newcommand{\Lonew}{L^{1,\infty}}

\newcommand{\one}{1}

\newcommand{\BMO}{\mathrm{BMO}}

\newcommand{\Maximal}{\mathcal{M}}

\newcommand{\AtwoBalanced}{A_2^{b}}
\newcommand{\ApBalanced}{A_p^{b}}
\newcommand{\ApprimeBalanced}{A_{p'}^{b}}
\newcommand{\AoneBalanced}{A_1^{b}}

\newcommand{\HS}{\mathrm{HS}}

\newcommand{\Cpb}{c_p^b}
\newcommand{\Coneb}{c_1^b}
\newcommand{\Ctwob}{c_2^b}

\newtheorem{theorem}{Theorem}[section]

\newtheorem{ltheorem}{Theorem}

\theoremstyle{definition}
\newtheorem{definition}[theorem]{Definition}

\newtheorem{lem}[theorem]{Lemma}
\newtheorem{prop}[theorem]{Proposition}

\newtheorem{cor}[theorem]{Corollary}
\newtheorem{rem}[theorem]{Remark}

\theoremstyle{remark}
\newtheorem{remark}[theorem]{Remark}

\numberwithin{equation}{section}

\begin{document}

\title[Sparse domination for balanced measures]{Balanced measures, sparse domination and complexity-dependent weight classes}


\author{Jos\'{e} M. Conde Alonso}
\address{Jos\'{e} M. Conde Alonso \hfill\break\indent 
 Departamento de Matem\'aticas \hfill\break\indent 
 Universidad Aut\'onoma de Madrid \hfill\break\indent 
 C/ Francisco Tom\'as y Valiente sn\hfill\break\indent 
 28049 Madrid, Spain}
\email{jose.conde@uam.es}
\author{Jill Pipher}
\address{Jill Pipher \hfill\break\indent 
 Department of Mathematics \hfill\break\indent 
 Brown University \hfill\break\indent 
 151 Thayer Street\hfill\break\indent 
 Providence, RI, 02912 USA}
\email{jill\_pipher@brown.edu}
\author{Nathan A. Wagner}
\address{Nathan A. Wagner \hfill\break\indent 
 Department of Mathematics \hfill\break\indent 
 Brown University \hfill\break\indent 
 151 Thayer Street\hfill\break\indent 
 Providence, RI, 02912 USA}
\email{nathan\_wagner@brown.edu}

\maketitle

\date{today}

\begin{abstract}
    We study sparse domination for operators defined with respect to an atomic filtration on a space equipped with a general measure $\mu$. In the case of Haar shifts, $L^p$-boundedness is known to require a weak regularity condition, which we prove to be sufficient to have a sparse domination-like theorem. Our result allows us to characterize the class of weights where Haar shifts are bounded. A surprising novelty is that said class depends on the complexity of the Haar shift operator under consideration. Our results are qualitatively sharp.  
\end{abstract}

\section*{Introduction}

Probabilistic techniques and methods are ubiquitous in harmonic analysis and especially so in Cal\-de\-r\'on-Zygmund theory. Often these methods make use of the filtration of $\sigma$-algebras generated by the dyadic system of cubes
$$
\D := \bigcup_{k\in\Z} \D_k = \left\{2^{-k}[j_1,j_1+1) \times [j_2,j_2+1) \times \ldots [j_n,j_n+1) : j \in \Z^n \right\}.
$$
Probabilistic stopping-time arguments have a long history as a critical tool in analysis; for example, in the Calder\'on-Zygmund decomposition, which is used to prove a weak-type estimate for singular integrals. Some fundamental insights into the close connection between certain function spaces and 
operators in analysis and their probabilistic counterparts were established in: \cite{GJ1982}, which explored the relationship
between BMO and dyadic BMO, and later in the dyadic representation theorems of \cites{P2000,Hyt2012}. 
In particular, the main result in \cite{P2000} recovers the Hilbert transform as an average of dyadic operators:
\begin{equation}\label{eq:PetermichlRep} \tag{HRep}
    \mathcal{H}f(x) = c_0 \int_\Omega \Sh_{\D^\omega}f(x) dP(\omega),
\end{equation}
where $\mathcal{H}$ is the Hilbert transform, $(\Omega,P)$ is a probability space, $\mathscr{D}^\omega$ is a different dyadic system for each value of $\omega$ and the dyadic Hilbert transform is the operator
\begin{equation}\label{eq:DyadicHilbert} \tag{HDyad}
\Sh_{\D^\omega}f(x) = \sum_{I \in \D^\omega} \langle f, h_I\rangle(h_{I_-}(x)-h_{I_+}(x)),   
\end{equation}
where $h_I=2^{-1}(|I_-|^{-1/2}\one_{I_-}-|I_+|^{-1/2}\one_{I_+})$ denotes the Haar function associated with the interval $I$, and $I_-$ (resp. $I_+$) is the left (resp. right) dyadic child of $I$. The main result in \cite{Hyt2012} expands on \eqref{eq:PetermichlRep} by showing that any Calder\'on-Zygmund operator on $\R^n$ can be written as an average of higher-dimensional generalizations of $\Sh$. 

Once it has been established a continuous operator can be recovered from probabilistic counterparts, one can often more readily derive quantitative consequences, like weighted estimates. One efficient way of doing so is the sparse domination technique, initiated in \cite{Ler2013} and applied to get pointwise estimates for Calder\'on-Zygmund operators via estimates for the dyadic objects that represent that operator \cites{CR2016,LN2019}. Sparse domination is not the only possible path: direct estimates are possible, like those in  \cite{Lacey2017} or using the one-third trick as in \cite{Mei2003}. 

The setting of the results discussed above is $\R^n$, equipped with the Lebesgue measure, which is a natural one for the problems under study. However, other measures $\mu$ on $\R^n$ also appear naturally, arising from questions in geometry or partial differential equations.
For example, the class of representing measures associated to elliptic or parabolic divergence form operators are all ``doubling". A measure $\mu$ is said to be doubling if there exists a constant $c_\mu$ such that for all balls $B$ centered at its support, there holds
\begin{equation} \label{eq:doubling} \tag{Db}
    \mu(2B) \leq c_\mu \mu(B).
\end{equation}
Another important class of measures is the class of measures of {\it polynomial growth}. These may fail \eqref{eq:doubling}. Despite the lack of the doubling property, using mainly geometric techniques, a very complete Calder\'on-Zygmund theory has been developed 
for measures satisfying the polynomial growth condition in \cites{NTV1998,To2001,To2001b,NTV2003}. These developments led to profound applications, among which are: the solution of the Painlev\'e problem \cite{To2003}, the $n-1$ case of the David-Semmes problem \cite{NTV2014}, and the characterization of the rectifiability of harmonic measure \cite{AHMMMTV2016}.

In martingale language, \eqref{eq:doubling} is closely related to the regularity of the filtration.
Regularity of the filtration is equivalent, in the case of the dyadic filtration on $(\R^n,\mu)$, to the condition $\mu(Q) \sim \mu(\widehat{Q})$ for all cubes in $\D$, where $\hat{Q}$ is the ``parent" dyadic cube of $Q$. Probabilistic tools are usually versatile enough to be applicable to $\sigma$-finite filtered spaces with highly non-regular structure, where the most important martingale inequalities hold true \cites{BG1970,Dav1970}. Therefore, one might be tempted to think that the techniques and results described above are still powerful even when the geometric and measure theoretic properties of the ambient space are not tightly coupled, even if \eqref{eq:doubling} fails. However, when the classical doubling condition of the underlying measure fails, the situation becomes significantly more complicated. The one-third trick is no longer useful in the usual way, and there are limited alternatives available \cite{C2020}. Moreover, direct arguments of sparse domination for Calder\'on-Zygmund operators \cites{CP2019,VZ2018} yield quantitative estimates that seem far from tight. To start, just as in the doubling setting, they do imply weighted inequalities for Calder\'on-Zygmund operators. Recall that a weight is an a.e. positive function $w\in L_{\mathrm{loc}}^1(d\mu)$ that we identify with the measure $d\nu=w d\mu$. The results in \cites{CP2019,VZ2018} imply that a Calder\'on-Zygmund operator is bounded on $L^2(w d\mu)$ whenever $w\in A_2(\mu)$, that is, the following holds:
\begin{equation}\label{eq:A2cond} \tag{A2}
    [w]_{A_2(\mu)} :=\sup_{B \; \mathrm{ball}} \frac{w(B)w^{-1}(B)}{\mu(B)^2} < \infty.
\end{equation}
However, in \cite{To2007}, the necessary and sufficient condition that weights must satisfy so that Calder\'on-Zygmund operators are bounded on $L^2(w d\mu)$, termed $\mathcal{Z}_2(\mu)$, was shown to be strictly weaker than the classical $A_2(\mu)$ condition. Moreover, the issue of a representation theorem along the lines of \eqref{eq:PetermichlRep} or \cite{Hyt2012} remained wide open. In that direction, the study of the operators that arise in the representation formulas is an interesting problem that can inform what, if any, representation theorems one can expect when the measure under consideration is not Lebesgue. Exploring the feasibility of extending this program to the non-doubling setting was one of the motivations of this 
work. We begin this study on $\R$, and point out that there will be additional considerations needed to extend the forthcoming results to higher dimensions.

\subsection*{Sparse domination for Haar shifts with respect to $\mu$} Let $\mu$ be a Radon, atomless measure on $\R$. We use the standard notation for dyadic ancestors: $I^{(j)}$ is the interval from $\D$ that contains $I$ and has sidelength equal to $2^j \ell(I)$. We denote
$$
\D_j(I)=\{J \in \D: J^{(j)} = I\}, \;\mbox{ and also } \D_{\leq j}(I)= \bigcup_{k=0}^j \D_k(I).
$$
Given an $I \in \mathscr{D}$, the Haar function associated to $I$ is constant on the dyadic children of $I$, normalized in $L^2(\mu)$, and has mean value zero:
$$
h_I(x) = \sqrt{\frac{\mu(I_-)\mu(I_+)}{\mu(I)}}\left(\frac{\one_{I-}(x)}{\mu(I_-)}-\frac{\one_{I+}(x)}{\mu(I_+)}\right).
$$
Given nonnegative integers $s$ and $t$, a Haar shift of complexity $(s,t)$ is an operator of the form
$$
T^{s,t}f(x)=\sum_{I \in \D} \sum_{J \in \D_r(I)} \sum_{K \in \D_s(I)} \alpha_{J,K}^{I}  \langle f, h_J \rangle h_K,
$$
where the pairing $\langle \cdot, \cdot\rangle$ (and integral averages $\langle \cdot \rangle_I $) are taken with respect to $\mu$:
$$
\langle f,g\rangle = \int_{\R^n} f(x) g(x) d\mu(x), \quad \langle f \rangle_Q = \frac{1}{\mu(Q)} \int_Q f(x) d\mu(x).
$$
In this language, $\Sh$ is an operator of complexity $(0,1)$. If $\alpha \in \ell^\infty$, then $T^{s,t}$ is bounded on $L^2(\mu)$ regardless of the properties of $\mu$. We will always assume that $\|\alpha\|_{\ell^\infty} \leq 1$. In \cite{LMP2014}, a condition is identified that characterizes $L^p(\mu)$ and weak type $(1,1)$ boundedness of $T^{s,t}$ when $r$ and $s$ are nonzero: if 
$$
m(I) = \frac{\mu(I_-)\mu(I_+)}{\mu(I)},
$$
then $T^{s,t}$ is of weak type $(1,1)$ if and only if, for all $I\in\D$,
\begin{equation} \label{eq:mBalanced} \tag{Balance}
    m(I) \sim m(\widehat{I}),
\end{equation}
where $\widehat{I}$ denotes the dyadic parent of $I$. 
When \eqref{eq:mBalanced} holds and $\mu$ is free of atoms, we will say that $\mu$ is \emph{balanced}\footnote{In \cite{LMP2014}, the authors term these measures $m$-equilibrated, but we prefer the shorter term {\it balanced}.}. The first question that we address in this paper is whether the operators $T^{s,t}$ satisfy sparse bounds for balanced measures $\mu$. In the case of the Lebesgue measure, this is contained in \cites{Hyt2012,CR2016}. We first consider estimates in the dual sense: we look for an inequality like
\begin{equation} \label{eq:sparseNo} \tag{ClasSp}
    \langle T^{s,t}f_1,f_2 \rangle \lesssim \sum_{I \in \Ss} \langle f_1 \rangle_Q \langle f_2 \rangle_Q \mu(Q) =: \A_\Ss(f_1,f_2),
\end{equation}
for nice enough $f_1$ and $f_2$ and a family $\Ss=\Ss(f_1,f_2) \subset \D$ which is sparse in the usual sense: for each $I\in \Ss$, there exists $E_I \subset I$ such that $\mu(E_I) \sim \mu(I)$ and the family $\{E_I\}_{I \in \Ss}$ is pairwise disjoint. The question is natural, since sparse domination methods are usually well adapted to dyadic operators. Notably, the simple argument in \cite{CDO2018} might be expected to go through in the balanced case once the classical Calder\'on-Zygmund decomposition employed there is replaced by the one in \cite{LMP2014}, (or the streamlined version from \cite{CCP2022}). In this direction, our first result is a negative one: we shall show that the classical sparse domination inequality,  \eqref{eq:sparseNo}, must fail, by constructing a balanced measure $\mu$ on the interval $[0,1]$ so that $\Sh$ fails sparse domination. Of course, similar examples can be constructed to disprove sparse domination in general for higher complexity dyadic shifts. 

\subsection*{Modified sparse forms} The above discussion shows that sparse domination, in the classical formulation, must fail for some balanced measures $\mu$. To remedy that, we propose a variant of the sparse form in the following way:
$$
\C_{\Ss}(f_1,f_2) := \sum_{I \in \Ss} \sum_{J \in \Ss \cap \D_{\leq 2}(\widehat{I})} \left[\langle f_1 \rangle_I \langle f_2 \rangle_{J} + \langle f_2 \rangle_I \langle f_1 \rangle_{J}\right] m(I).
$$
The role of the second sum in the definition of $\C_\Ss$ takes into account the interaction of intervals $I$ and its dyadic neighbors, which is necessary for the sparse domination argument to work when the complexity is nonzero. On the other hand, replacing $\mu(I)$ by $m(I)$ on the right hand side helps one to obtain quantitative bounds for $\C_\Ss$, since we have $m(I) \ll \mu(I)$ in general. With this definition, the result that we obtain reads as follows: 
\begin{ltheorem} \label{th:thmA}
Let $\mu$ be balanced, and let $T$ be a Haar shift of complexity $(s,t)$ with $s + t \leq 1$. Then for each pair of compactly supported, bounded nonnegative functions $f_1,f_2$ there exists a sparse collection $\Ss \subset \D$ such that 
$$
| \langle Tf_1,f_2\rangle | \lesssim \A_\Ss(f_1,f_2) + \C_\Ss(f_1,f_2).
$$
\end{ltheorem}

We postpone the full statement of Theorem \ref{th:thmA} for Haar shifts of higher complexity to Section \ref{sec:section2}. As is standard, the proof of Theorem \ref{th:thmA} uses the weak $(1,1)$ bound for the maximal operator $\Maximal$ associated
to the dyadic grid $\D$. On the other hand, both forms $\A_\Ss$ and $\C_\Ss$ satisfy the expected weak-type and $L^p(\mu)$ estimates, which is evidence that this definition of the sparse forms is the right one in the context of balanced measures. The broad scheme of the proof of theorem \ref{th:thmA} is similar to arguments that were used in the context of doubling measures like \cites{BFP2016,CCDO2017}, and is most similar to \cite{CDO2018}. However, we need to use a Calder\'on-Zygmund decomposition adapted to general measures, originally inspired by Gundy's decomposition for martingales \cites{LMP2014,CCP2022}. An additional feature that we need to use is a $\BMO$ estimate for the good part of the decomposition, which is a technical novelty key to our approach.

\subsection*{A weighted theory for balanced measures} The main applications of sparse bounds are weighted estimates. The classical sparse domination yields quantitatively sharp estimates for doubling measures. In the case of the Lebesgue measure, both Calder\'on-Zygmund operators and Haar shifts are bounded on $L^2(w)$ if and only if $w \in A_2$. The corresponding condition with respect to a general measure $\mu$ is just \eqref{eq:A2cond}, where we should take the sup with respect to all $Q \in \D$ in the Haar shift case (instead of over balls). As we mentioned before, in Calder\'on-Zygmund case the $A_2(\mu)$ condition is not necessary in general, but it is certainly sufficient \cite{To2007}. In the dyadic setting, the situation is different: we can show that the $A_2(\mu)$ condition is not sufficient for the boundedness of $\Sh$, by constructing a pair $(\mu,w)$ where $\mu$ is balanced, $w\in A_2(\mu)$ and $\Sh$ is not bounded on $L^2(w d\mu)$. Our main weighted result is the identification of the right class of weights that governs the boundedness of Haar shifts. A notable feature of this identification is the fact that it takes into account the higher complexity of the operators that we consider. Moreover, unlike the sparse form, the weight class is the same for all operators of complexity greater than or equal to 1. 

\begin{ltheorem} \label{th:thmB}
Let $\mu$ be balanced. There exists a collection of weights $\AtwoBalanced$ (balanced $A_2$) with the following properties:
\begin{itemize}
    \item $\AtwoBalanced \subseteq A_2(\mu)$.
    \item If $(s,t) \neq (0,0)$, all operators $T^{s,t}$ are bounded on $L^2(w d\mu)$ if and only if $w \in \AtwoBalanced$. 
\end{itemize}
Moreover, there exists a balanced, nondoubling measure $\mu$ so that $\AtwoBalanced \subsetneq A_2(\mu).$
\end{ltheorem}
We postpone, until section \ref{sec:section3}, the precise definition of the class $\AtwoBalanced(\mu)$, which can be described as a variant of 
\eqref{eq:A2cond} involving averages over dyadic intervals $I, J$ where $J \in \D_{\leq 2}(\widehat{I}) $.
  As one would expect, in case $\mu$ is doubling, we have that $\AtwoBalanced(\mu)=A_2(\mu)$. Our weighted results can be readily generalized to $L^p(w d\mu)$, with a straightforward modification of our $A_2$-type class. Finally, we have a maximal function characterization. We formulate it by identifying a family of variants of the usual dyadic maximal function, that we denote by $\Maximal^j$. $\Maximal^j$ is $L^p(\mu)$-bounded and weak-type $(1,1)$, and satisfies the same sparse domination as the Haar shifts $T^{s,t}$ whenever $r+s=j$. We then use each $\Maximal^j$ to define a family of $A_p$-type classes of weights: we say that $w\in A_1^j(\mu)$ if $\Maximal^jw(x) \lesssim w(x)$, while we say that $w\in A_p^j(\mu)$ if $\Maximal^j:L^p(w d\mu) \to L^p(w d\mu)$. Even though the operators $\Maximal^j$ are not equivalent, the weight classes that they determine are the same; that is, $A_p^j(\mu) = \ApBalanced(\mu)
$
for all $j \geq 1$.

The rest of the paper is organized as follows: Section \ref{sec:section1} contains an example that shows that sparse domination in the usual sense must fail for Haar shifts and some balanced measures, and another one of $w \in A_2(\mu)$ |with $\mu$ balanced| so that $\Sh$ is not bounded on $L^2(w d\mu)$. Section \ref{sec:section2} contains the proof of Theorem \ref{th:thmA} and our sparse domination result for general complexity shifts on $\R$. Section \ref{sec:section3} contains the weighted theory, including the proof of Theorem \ref{th:thmB}.

\subsection*{Acknowledgements} The authors would like to thank Sergei Treil for several valuable discussions. The first named author has been supported by grants CNS2022-135431 and RYC2019-027910-I (Ministerio de Ciencia, Spain). The third named author is supported by grant DMS-2203272 (National Science Foundation).

\section{Failure of usual sparse domination} \label{sec:section1}

We shall work with Borel measures $\mu$ on $\R$ that we will always assume to be balanced. This implies for us that $\mu$ has no atoms -which we shall need in order to use the Carleson packing characterization of the sparse condition below-. Fix $0<\eta<1$. A collection $\Ss \subset \D$ is called $\eta$-sparse (or just sparse) if for each $I \in \Ss$ there exists $E_I \subset I$ so that $E_I \cap E_{I'} = \emptyset$ if $I\neq I'$ and $\mu(E_I) \geq \eta \mu(I)$. According to \cite[Theorem 1.3]{Han2018}, the family $\Ss$ is sparse if and only if it satisfies a the Carleson packing condition
\begin{equation} \label{eq:CarlesonPacking}
\sum_{\substack{J \subset I:\\ J \in \Ss}} \mu(J) \leq C_\mu \mu(I)
\end{equation}
for all $I \in \D$. One important consequence of \eqref{eq:CarlesonPacking} is the fact that the union of two sparse families is again sparse (with a possibly different constant), and this observation we shall use throughout. The sparse operator associated with $\Ss$ is
$$
\A_\Ss f(x) = \sum_{I \in \Ss} \langle f \rangle_I \one_I(x),
$$
while the corresponding bilinear form already appeared in the Introduction:
$$
\A_\Ss(f,g) = \langle \A_\Ss f,g\rangle = \sum_{I \in \Ss} \langle f \rangle_I \langle g \rangle_I \mu(I).
$$
We start constructing such a measure $\mu$ so that $\Sh$ (we omit the dyadic grid dependence from the notation) does not admit sparse domination in the usual dual sense.
 
\begin{prop}\label{prop:SparseFailure} Fix $0 < \eta < 1$. There exists a balanced measure $\mu$ and two sequences of compactly supported functions $\{f_j\},\{g_j\} \subset \Ltwo(\mu)$ such that for all $\eta$-sparse families $\Ss \subset \D$ 
\begin{equation}\label{eq:sparsefailure}
|\langle \Sh f_j, g_j \rangle | \gtrsim j |\A_\Ss(f_j,g_j)|.    
\end{equation}
\end{prop}

 \begin{proof} We take the measure $\mu$ to be the one constructed in \cite{LMP2014}*{Section 4, (a)}. In particular, $\mu$ is supported on the unit interval\footnote{If we want to consider measures supported on all of $\mathbb{R}$, we can just take $\mu$ to be uniform with density $1$ outside $[0,1]$ and the argument still works.} and is defined inductively as follows: let $a_k=1-1/k$ and $b_k=1/k$ for $k \geq 2$. Set $I_k=[0,2^{-k}]$ for $k \geq 0$, and denote by $I_k^b=[2^{-k},2^{-k+1}]$ its dyadic sibling. Set $\mu(I_0)=1$, $\mu(I_1)=\mu(I_1^b)=\frac{1}{2}$ and for $k \geq 2,$ $\mu(I_k)= a_k \mu(I_{k-1})= a_k \mu(\widehat{I}_k)$, so $\mu(I_k^b)=b_k \mu(I_{k-1})$. The construction is finished by declaring that the density of $\mu$ is taken to be uniform on $I_k^b$. $\mu$ is non-dyadically doubling because 
$$
\frac{\mu(\widehat{I}_k)}{\mu(I_k^b)}=\frac{1}{b_k}=k \rightarrow \infty \mbox{ as } k \rightarrow \infty.
$$ 
Clearly, we have $\mu(I_k) \sim \mu(\widehat{I}_k)$ since the sequence $a_k$ is bounded from above and below. Straightforward computations lead to 
 $$m(I_k^b)= \frac{1}{4} b_k \mu(\widehat{I}_k) \sim \frac{\mu(\widehat{I}_k)}{k},\quad m(I_k)= a_{k+1} b_{k+1} a_k \mu(\widehat{I}_k) \sim \frac{\mu(\widehat{I}_k)}{k},$$
which show that $\mu$ is balanced. We also have the estimates
$$ 
\mu(I_k) \sim \frac{1}{k}, \quad \mu(I_k^b) \sim \frac{1}{k^2}.           $$
One can also show that $\mu$ does not have any point masses. Indeed, the only possible point mass is $\delta_0$, but this does not occur since $\mu(I_k) \rightarrow 0$ as $k \rightarrow \infty$ and so $\mu(\{0\})=0$. To prove \eqref{eq:sparsefailure}, we assume $\eta \in (0,1)$ has been fixed and suppress dependence on $\eta$ in the computations that follow. For $j \geq 2$, we take $f_j= \one_{I_{j-1}^b}$ and $g_j= \one_{I_{j}^b}.$ As a preliminary, we notice $f_j$ and $g_j$ have disjoint supports and if $I \in \D$ intersects both $\text{supp} (f_1)$ and $\text{supp} (f_2)$ non-trivially, we must have $I=I_k$ with $0 \leq k \leq j-2.$ We also observe

 \begin{equation} \Maximal f_j(x) = \frac{\mu(I_{j-1}^b)}{\mu(I_{k-1})}, \,  \Maximal g_j(x) =\frac{\mu(I_{j}^b)}{\mu(I_{k-1})}, \, x \in I_{k}^b \text{ for }  0 < k \leq  j-2; \label{MaximalPointwiseBound1} \end{equation}

  \begin{equation} \Maximal f_j(x) = 1, \,  \Maximal g_j(x) = \frac{\mu(I_{j}^b)}{\mu(I_{j-2})}, \, x \in I_{j-1}^b ; \label{MaximalPointwiseBound2} \end{equation}

    \begin{equation} \Maximal f_j(x) = \frac{\mu(I_{j-1}^b)}{\mu(I_{j-2})}, \,  \Maximal g_j(x) = 1, \, x \in I_{j}^b ; \label{MaximalPointwiseBound3} \end{equation}

 \begin{equation}  \Maximal f_j(x) = \frac{\mu(I_{j-1}^b)}{\mu(I_{j-2})}, \,  \Maximal g_j(x) =\frac{\mu(I_{j}^b)}{\mu(I_{j-1})}, \, x \in I_{k}^b \text{ for }  k > j  \label{MaximalPointwiseBound4}  \end{equation}

 Using the sparsity of $\Ss$ and equations \eqref{MaximalPointwiseBound1} through \eqref{MaximalPointwiseBound4}, we can estimate the sum directly:

 \begin{align*}
 \sum_{I \in \Ss}  \langle f_j \rangle_{I}  \langle g_j \rangle_{I} \mu(I)  & \lesssim \sum_{I \in \Ss}  \langle f_j \rangle_{I}  \langle g_j \rangle_{I} \mu(E_I) \\
& \leq  \int_{[0,1]} \Maximal f_j \cdot  \Maximal g_j\, d\mu   \\
 & \leq \sum_{k=1}^{\infty} \int_{I_k^b} \Maximal f_j \cdot  \Maximal g_j \, d\mu \\
 & = \sum_{k=1}^{j-2} \left( \frac{\mu(I_{j-1}^b)}{\mu(I_{k-1})} \right) \cdot \left(\frac{\mu(I_{j}^b)}{\mu(I_{k-1})}\right)  \cdot   \mu(I_k^b) + \left( \frac{\mu(I_j^b)}{\mu(I_{j-2})}\right) \cdot \mu(I_{j-1}^b)\\
 & +  \left( \frac{\mu(I_{j-1}^b)}{\mu(I_{j-2})}\right) \cdot \mu(I_{j}^b) + \sum_{k=j+1}^{\infty}  \left( \frac{\mu(I_{j-1}^b)}{\mu(I_{j-2})} \right) \cdot \left(\frac{\mu(I_{j}^b)}{\mu(I_{j-1})}\right)  \cdot   \mu(I_k^b) \\
 & \lesssim  \left(\frac{1}{j^4} \sum_{k=0}^{j-2} 1\right) + \frac{1}{j^3} + \frac{1}{j^3}+ \left(\frac{1}{j^2}\sum_{k=j+1}^{\infty} \frac{1}{k^2} \right) \lesssim \frac{1}{j^3}.
 \end{align*}

 Now we turn to a lower bound for $|\langle \Sh f_j, g_j \rangle|.$ Given $I \in \D$, denote by $S_I$ the simple shift $S_I f= \langle f, h_I \rangle (h_{I_{-}}-h_{I_{+}}).$ Then, by definition and using the support properties of $f_j$ and $g_j,$ we have $\langle S_I f_j, g_j \rangle= \langle f_j, h_I \rangle \cdot \langle h_{I_{-}}-h_{I_{+}}, g_j \rangle \neq 0 $ if and only if $I= I_k$ with $0 \leq k \leq j-2.$ Therefore, we may write
 $$
 \langle \Sh f_j, g_j \rangle = \sum_{k=0}^{j-2} \langle S_{I_k} f_j, g_j \rangle = \sum_{k=0}^{j-2}  \langle f_j, h_{I_k} \rangle \cdot \langle h_{I_{k+1}}- h_{I_{k+1}^b}, g_j \rangle.
 $$
For the last term in the sum, we have 
$$\langle f_j, h_{I_{j-2}} \rangle = -\sqrt{m(I_{j-2})} \frac{\mu(I_{j-1}^b)}{\mu(I_{j-1}^b)}=-\sqrt{m(I_{j-2})},
$$ 
while
 $$
 \langle h_{I_{j-1}}- h_{I_{j-1}^b}, g_j \rangle = \langle h_{I_{j-1}}, g_j \rangle= - \sqrt{m(I_{j-1})}.
 $$

 Therefore, $$\langle f_j, h_{I_{j-2}} \rangle \cdot \langle h_{I_{j-1}}- h_{I_{j-1}^b}, g_j \rangle= \sqrt{m(I_{j-1})m(I_{j-2}}) \sim \frac{\mu(I_j)}{j} \sim \frac{1}{j^2}.$$
 Notice that all the other terms in the summation will similarly be positive, because if $k<j-2,$ $f_j$ is supported entirely on the left half of $I_k$ and $g_j$ is supported on the left half of $I_{k+1}.$ Therefore, we need not consider the other terms and \eqref{eq:sparsefailure} follows immediately. 
\end{proof}

\begin{rem}
If one were instead to consider the measure $\mu'$ defined by $b_k'=\frac{1}{k^2}$ with parallel construction to Proposition \ref{prop:SparseFailure}, then $\mu'$ would have a point mass at $0$. In this case one can actually prove the much stronger result that the bilinear form $|\langle \Sh f_1, f_2 \rangle|$ cannot be controlled by $ \sum_{I \in \Ss} \langle f_1 \rangle_I \langle f_2 \rangle_I \mu(I)$ for \emph{any} dyadic collection $\Ss$, let alone a sparse one.
\end{rem}

Our next construction uses the same measure $\mu$ that we used in Proposition \ref{prop:SparseFailure} to show that the class $A_2(\mu)$ is not the right one to study $\Ltwo$-boundedness for higher complexity shifts. Again, we exemplify using $\Sh$. 

\begin{prop}\label{BadWeight}
There exists a balanced measure $\mu$ and a weight $w \in A_2(\mu)$ so that $\Sh$ is unbounded on $\Ltwo(w  d \mu)$.    
\end{prop} 

\begin{proof}
Define $\mu$ as in the proof of Proposition \ref{prop:SparseFailure}. Define the weight $w$ as follows: 
$$
w(x)= \begin{cases} 2^{-\frac{k}{2}}, & x \in I_{2^k}^b \text{ for $k \geq 1$,} \\ 1, & \text{otherwise.} \end{cases}
$$

We claim $w \in A_2(\mu)$. To see this, first note that if $I$ is contained in $I_{k}^b$ for some $k$, then $\langle w \rangle_{I} \langle w^{-1} \rangle_{I}=1.$ The only other case to consider is $I=I_\ell$ for some $\ell$. We will use the simple facts that $\sum_{j > \ell}^{\infty} \mu(I_j^b)= \mu(I_{\ell})$ and $\mu(I_j) \leq \mu(I_k)$ if $j \geq k.$  We estimate
\begin{align*}
\langle w^{-1} \rangle_{I_\ell} & = \frac{ \sum_{ \substack{j> \ell: \\ j \neq 2^k \text{ for any $k$} }} \mu(I_{j}^b)+ \sum_{ \substack{k:\\ 2^k > \ell }} 2^{\frac{k}{2}} \mu(I_{2^k}^b)}{\mu(I_\ell)}\\
& \leq 1+  \frac{ \sum_{ \substack{k:\\ 2^k > \ell }} 2^{-\frac{k}{2}} \mu(I_{2^k-1})}{\mu(I_{\ell})} \lesssim 1,
\end{align*} 
using $\ell\mu(I_\ell)=\mu(I_{\ell+1})$. The implicit constant is independent of $\ell.$ The estimate $ \langle w \rangle_{I_{\ell}} \lesssim 1$ is even easier, and so $\langle w \rangle_{I_\ell} \langle w^{-1} \rangle_{I_\ell} \lesssim 1$ in this case as well, yielding the claim. However, notice that  
$$
\langle w^{-1} \rangle_{I_{2^k}^b} \langle w \rangle_{I_{2^k+1}^b}=2^{\frac{k}{2}} \rightarrow \infty \mbox{ as } k \rightarrow \infty.
$$ 
This behavior of the weight will lead to a contradiction: consider test functions $f_k \equiv w^{-1} \one_{I_{2^k}^b}$ and $g_k \equiv \one_{I_{2^k+1}^b}.$ We compute 
$$
\|f_k\|_{\Ltwo(w d\mu)}= \left(\int_{I_{2^k}^b} w^{-1} \, d\mu\right)^{\frac{1}{2}}= 2^{\frac{k}{4}} \mu(I_{2^k}^b)^{\frac{1}{2}},
$$ 
and similarly we have 
$$
\|g_k\|_{\Ltwo(w d\mu)}= \mu(I_{2^k+1}^b)^{\frac{1}{2}}.
$$ 
We now obtain a lower bound for $|\langle \Sh f_k, w g_k \rangle |$ using similar considerations as in the proof of Proposition \ref{prop:SparseFailure}:

\begin{align*}
|\langle \Sh f_k, w \, g_k \rangle | & = \left| \sum_{j=0}^{2^k-1} \langle S_{I_j} f_k, w \, g_k  \rangle \right| \\
& \geq |\langle f_k, h_{I_{2^k-1}}\rangle| \cdot | \langle h_{I_{2^k}}- h_{I_{2^k}^b} , w  \, g_k  \rangle|\\
& = 2^{\frac{k}{2}} \sqrt{m(I_{2^k-1})} \cdot \sqrt{m(I_{2^k})}\\
&  \sim 2^{\frac{k}{2}} \mu(I_{2^k}^b)^{\frac{1}{2}} \mu(I_{2^k+1}^b)^{\frac{1}{2}}\\
& = 2^{\frac{k}{4}} \|f_k\|_{\Ltwo(w d\mu)} \|g_k\|_{\Ltwo(w d\mu)}. 
\end{align*}
This shows that $\Sh$ is unbounded on $\Ltwo(w d\mu).$
\end{proof}

\section{Modified Sparse Domination and Consequences} \label{sec:section2}


\subsection{Modified Sparse forms and Calder\'on-Zygmund estimates}

Proposition \ref{prop:SparseFailure} shows that usual sparse forms are not enough to achieve sparse domination for Haar shifts in general. To remedy this, we introduce other forms that reflect the complexity of the operator that we want to dominate. The dyadic distance between two intervals $I, J \in \D$ that share a common ancestor is defined as follows:
\begin{equation*} 
\dyad(I,J):= \min_{(s,t): I^{(s)}= J^{(t)}} (s+t).       \end{equation*}
If $I,J$ do not share a common ancestor, then by convention we let $\dyad(I,J)=\infty.$ Notice that $\dyad(I,J)=0$ if and only if $I=J$. For any $I \in \D$, we let 
$C_N$ the number of $J\in \D$ such that $\dyad(I,J) \leq N+2$. $C_N$ does not depend on $I$. For convenience, we enumerate the dyadic intervals $J \in \D$ that satisfy  $2<\dyad(I,J) \leq N+2$ and $J \cap I = \emptyset$  as $c_1(I),c_2(I),\cdots, c_{N'}(I)$, where $N'$ is a constant depending only on $N$. We will suppose that we have fixed an arbitrary enumeration function for each $I \in \D$ and abuse notation by writing $c_j$ to represent all enumerations for all intervals. Note that given a fixed $J \in \D$, there are at most $N'$ intervals $I$ satisfying $c_j(I)=J$ for some $j$. This fact will be used in various estimates without further comment. Given $\Ss \subset \D$, $N \in \mathbb{N}$, and $f_1,f_2 \in \Ltwo(\mu)$, we define 
\begin{equation*} 
\mathcal{C}_{\Ss}^N(f_1,f_2):= \sum_{\substack{J, K \in \Ss:\\ 2<\dyad(J,K) \leq N+2,\\ J \cap K= \emptyset}} \langle f_1 \rangle_{J} \langle f_2 \rangle_{K}\, \, \sqrt{m(J)} \sqrt{m(K)}.
\end{equation*} 

We will later show that the forms $\mathcal{C}_{\Ss}^N$ are $\Lp$-bounded and of weak-type $(1,1)$. What we show next is a technical lemma that justifies considering only intervals at distance greater than $2$ and disjoint in the definition.

\begin{lem}\label{lem:SparseReduction}
Let $N\in\mathbb{N}$ and $\Ss \subset \D$ be $\eta$-sparse. There exist $0<\eta'\leq \eta$ and $\Ss' \supseteq \Ss$ $\eta'$-sparse such that 
$$
    \sum_{\substack{J, K \in \Ss:\\ 1 \leq \dyad(J,K) \leq N+2}} \langle f_1 \rangle_{J} \langle f_2 \rangle_{K}\, \, \sqrt{m(J)} \sqrt{m(K)} \lesssim \mathcal{A}_{\mathcal{S'}}(f_1, f_2)+  \mathcal{C}_{\mathcal{S'}}^N(f_1,f_2)
$$



\end{lem}

\begin{proof}
Since we are going to construct $\Ss'$ so that $\Ss \subseteq \Ss'$, we only need to bound terms arising from pairs $(J,K)$ such that either $\dyad(J,K)=2$ and $J \cap K=\emptyset$, or $J \cap K \in \{J,K\}$ and $\dyad(J,K)\leq N+2$. In the latter case, if $J=K^{(k)}$ for some $k \leq N+2$ -say-, we estimate
$$
\langle f_1 \rangle_{J} \langle f_2 \rangle_{K}\, \, \sqrt{m(J)} \sqrt{m(K)} \lesssim_k \langle f_1 \rangle_{J} \, \langle f_2 \rangle_{J} \, \, \mu(J),
$$
because $m(J) \sim m(K)$ and $m(J) \leq \mu(J)$. In the former case, we must have $K=J^b$. If we suppose $J$ and $J^b$ both belong to $\Ss$, then

$$\langle f_1 \rangle_{J} \langle f_2 \rangle_{J^b}\, \, \sqrt{m(J)} \sqrt{m(J^b)} \lesssim \langle f_1 \rangle_{\widehat{J}} \, \langle f_2 \rangle_{\widehat{J}} \, \, \mu(\widehat{J}),$$
where the implicit constant is independent of $J$ because each dyadic interval has exactly two children. We define now $\Ss'$ as the union of the family $\Ss$ and the intervals $\widehat{J}$ for each $J \in \Ss$ such that $J^b \in \Ss$. It can be seen, using \eqref{eq:CarlesonPacking} that $\Ss'$ is $\eta'$-sparse for some $\eta'=\eta'(\eta,N)$, and the assertion follows.
\end{proof}

To prove Theorem \ref{th:thmA} we will perform Calder\'on-Zygmund decompositions at different heights for two functions $f_1$ and $f_2$ simultaneously. In the Lebesgue measure case, the Calder\'on-Zygmund decomposition yields $L^\infty$ estimates for the good part of the function, but that is unavailable in our case, so we must content ourselves with a $\BMO$ estimate. We use the standard martingale $\BMO$ norm (associated with the filtration generated by $\D$), which is 
$$
\|f\|_{\BMO}:=\sup_{I \in \D} \frac{1}{\mu(I)} \int_{I} |f -\langle f \rangle_{\widehat{I}}| \, d\mu.
$$                    

\begin{lem}\label{CZD} Let $0\leq f_1, f_2 \in \Lone(\mu)$ be supported on $I \in \D$ and $\lambda_1, \lambda_2>0$. Then there exist functions $g_j,b_j$ such $f_j=g_j+b_j$ for $j=1,2$ and satisfying the following properties:

\begin{enumerate}
\item \label{item:1} There exist pairwise disjoint, dyadic intervals $\{I_k\} \subset \D(I)$ so that $b_j= \sum_{k=1}^{\infty}b_{j,k}$ where each $b_{j,k}$ is supported on $\widehat{I}_k$, $\int_{\widehat{I}_k}b_{j,k}=0$, and $\|b_{j,k}\|_{\Lone(\mu)} \lesssim \int_{I_k} |f_j| \,d \mu$ for $j=1,2.$ In particular, for $j=1,2$ we have
$$b_{j,k}= f_j \one_{I_k}- \langle f_j \one_{I_k} \rangle_{\widehat{I}_k} \one_{\widehat{I}_k}.$$
\item \label{item:2} For $j=1,2$, the function $g_j \in \Lp(\mu)$ for all $1 \leq p< \infty$ and satisfies $\|g_j\|_{\Lp(\mu)}^p \leq C_p \lambda_j^{p-1} \|f_j\|_{\Lone(\mu)}.$
\item \label{item:3} For $j=1,2$ the function $g_j \in \BMO$ and $\|g_j\|_{\BMO}\leq \lambda_j$.  

\end{enumerate}
\end{lem}

\begin{proof} The statement of the decomposition is contained in \cites{LMP2014,CCP2022}, except for \eqref{item:3}, which is new and for which we give full details below. The intervals $I_k$ are selected using the usual stopping-time argument: let $\{I_k\}$ be the set of maximal dyadic intervals contained in $I$ so that either $\langle f_{1} \rangle_{I_k}> \lambda_1$ or $\langle f_{2} \rangle_{I_k} >\lambda_2$. Then \eqref{item:1} follows. We now turn to \eqref{item:3}, which by interpolation proves \eqref{item:2}. For $j=1,2$, we have
$$g_j= f_j \one_{\left(\cup_{k} I_k\right)^c}+\sum_{k} \langle f_j \one_{I_k} \rangle_{{\widehat{I}_k}} \one_{\widehat{I}_k} := g_{j,1}+g_{j,2}.$$
By the Lebesgue Differentiation theorem, for almost every point $x \in \left(\cup_{j} I_j\right)^c$ there is a sequence of unselected cubes shrinking down to $x$  and we have $|f_j(x)| \leq \lambda_j,$ so $\|g_{j,1}\|_{L^\infty}\leq \lambda_j.$ It easily follows that $\|g_{j,1}\|_{\BMO}\leq 2 \lambda_j$ for $j=1,2.$ It therefore suffices to show $\|g_{j,2}\|_{\BMO} \lesssim \lambda_j$ for $j=1,2.$ 
If $I=I_k$ is a selected interval, we have
\begin{align*}
\langle g_{j,2} \rangle_{\widehat{I}_k} \one_{I_k}& =  \left(\sum_{\ell: \widehat{I}_\ell \subsetneq \widehat{I}_k} \frac{\mu(\widehat{I}_\ell)}{\mu(\widehat{I}_k)} \langle f_j \one_{I_\ell} \rangle_{\widehat{I}_\ell} + \sum_{\ell: \widehat{I}_k \subset \widehat{I}_\ell} \langle f_j \one_{I_\ell} \rangle_{\widehat{I}_\ell} \right) \one_{I_k}\\
& = \frac{ \int_{E_k} f_j \, d \mu}{\mu(\widehat{I}_k)} \one_{I_k} + \left(\sum_{\ell: \widehat{I}_k \subset\widehat{I}_\ell} \langle f_j \one_{I_\ell} \rangle_{\widehat{I}_\ell}\right) \one_{I_k},
\end{align*}
where $$E_k= \bigcup_{\ell: \widehat{I}_\ell \subsetneq \widehat{I}_k} I_{\ell}.$$
On the other hand, notice that by definition and the disjointness of the selected intervals, $g_{j,2} \one_{I_k}= \left(\sum_{\ell} \langle f_j \one_{I_{\ell}} \rangle_{\widehat{I}_\ell}\one_{\widehat{I}_\ell}\right) \one_{I_k}=  \left(\sum_{\ell: \widehat{I}_k \subset\widehat{I}_\ell} \langle f_j \one_{I_\ell} \rangle_{\widehat{I}_\ell}\right) \one_{I_k} $ 
and therefore $$\frac{1}{\mu(I_k)} \int_{I_k}|g_{j,2}- \langle g_{j,2} \rangle_{\widehat{I}_k}| \, d \mu =  \frac{ \int_{E_k} f_j \, d \mu}{\mu(\widehat{I}_k)} \leq \langle f_j \rangle_{\widehat{I}_k} \leq \lambda_j.  $$
If $I$ is not a selected interval, we can similarly argue
$$\langle g_{j,2} \rangle_{\widehat{I}} \one_{I}= \frac{ \int_{E_I} f_j \, d \mu}{\mu(\widehat{I})} \one_{I} + \left(\sum_{\ell: \widehat{I}_ \subset\widehat{I}_{\ell}} \langle f_j \one_{I_{\ell}} \rangle_{\widehat{I}_\ell}\right) \one_{I},$$ where $E_I= \bigcup_{\ell: \widehat{I}_\ell \subsetneq \widehat{I}} I_{\ell}.$
However, in this case it is possible that $I= \widehat{I}_\ell$ for some $\ell,$ so we have $g_2 \one_{I}= \left(\sum_{\ell: \widehat{I}_ \subset\widehat{I}_\ell} \langle f_j \one_{I_\ell} \rangle_{\widehat{I}_\ell}\right) \one_{I}+ \left(\sum_{\ell:I= \widehat{I}_\ell} \langle f_j \one_{I_\ell} \rangle_{\widehat{I}_\ell}\right) \one_{I}  .$ However, there are at most two terms in the second summation, and each average in the sum is controlled by $\lambda_j$, so the $\BMO$ norm can be estimated as before. 
\end{proof}

 We next introduce notation about collections of intervals. Suppose functions $f_1,f_2$ and an initial interval $I_0 \in \D$ have been fixed as in Lemma \ref{CZD}. For $I \subseteq I_0$, let $\mathcal{B}(I)$ denote the collection of selected intervals in Lemma \ref{CZD} applied to $f_j \one_{I}$ at heights $\lambda_j= 16 \langle f_j \rangle_{I}$, $j=1,2$. We then inductively define, for $k \in \mathbb{N}$: 
\begin{equation} \label{eq:notationGB}
\mathcal{B}_0(I)= \mathcal{B}(I), \quad \mathcal{B}_1(I)= \bigcup_{J \in \mathcal{B}(I)} \mathcal{B}(J), \quad \mathcal{B}_k(I)=\bigcup_{J \in \mathcal{B}_{k-1}(I)} \mathcal{B}(J). 
\end{equation}
Moreover, given $N \in \mathbb{N}$, we set $$\mathcal{B}^N(I):= \bigcup_{k=0}^{N} \mathcal{B}_k(I).$$
Finally, we set
$$\mathcal{G}(I):= \{J \in \D(I): J \not \subset K \text{ for any } K \in \mathcal{B}(I)\}, \quad \mathcal{O}(I)=\{J \in \D: J \supsetneq I\}.$$

\subsection{Haar shifts} Let $\alpha=\{\alpha_{J,K}^I\}_{I, J, K \in \D}$ be a (triply indexed) sequence with $\|\alpha\|_\infty \leq 1$. For any two integers $s,t \in \mathbb{N}$, we define a \emph{Haar shift of complexity $(s,t)$} by
$$
T_\alpha^{s,t}f:= \sum_{I \in \D} \sum_{J \in \D^s(I)} \sum_{K \in \D^t(I)} \alpha_{J,K}^{I}  \langle f, h_J \rangle h_K, \quad f \in \Ltwo(\mu) 
$$
The above sum converges in $\Ltwo(\mu)$ and $\mu$-almost everywhere. The family of Haar shifts of complexity $(s,t)$ will be denoted by
$$
\HS(s,t) = \left\{T_\alpha^{s,t}: \|\alpha\|_{\infty} \leq 1 \right\}.
$$
It is immediate from the fact that $\{h_I\}_{I \in \D}$ forms an orthonormal basis for $\Ltwo(\mu)$ that for any complexity $(s,t)$ and any measure $\mu$, any $T\in \HS(s,t)$ is bounded on $\Ltwo(\mu)$ with operator norm only depending on the complexity. If $\mu$ is balanced, then $T$ is bounded on $\Lp(\mu)$ and of weak-type $(1,1)$, as shown in \cite{LMP2014}. We are now ready to state our main theorem. 

\begin{theorem}\label{th:ModifiedSparse}
Let $\mu$ be balanced and $N\in \N$. There exists $\eta \in (0,1)$ such that if $(s,t)$ satisfies $s+t \leq N$, the following holds: for all $T \in \HS(s,t)$ and each pair of compactly supported, bounded nonnegative functions $f_1,f_2$ there exists an $\eta$-sparse collection $\Ss \subset \D$ such that 
\begin{align*}
\mathlarger{|} \langle T  f_1, f_2 \rangle \mathlarger{|} & \lesssim \left( \mathcal{A}_{\Ss}(f_1, f_2)+  \mathcal{C}_{\Ss}^N(f_1,f_2) \right).
\end{align*}

\end{theorem} 

For the remainder of this subsection we assume that the pair $(s,t)$ has been fixed, and the two functions $f_1$ and $f_2$ are supported on a fixed dyadic interval $I_0$. We also suppose that the sequence $\alpha$ that defines the Haar shift $T$ is nonzero except for finitely many indices, since the estimates we obtain will be uniform over all finite sums, and that $\alpha$ has been fixed as well. It will be useful to introduce the following notation: let $I=[2^{-k} p, (p+1) 2^{-k}) \in \D$, where $k, p \in \mathbb{Z}.$ Given fixed $s, t \in \mathbb{N}$, let $0 \leq m \leq 2^{s}-1$ and $0 \leq n \leq 2^{t}-1$ be integers. We index the intervals $J \in \D^{s}(I)$ via
$$ I_{s}^m=[2^{-k}p+m 2^{-k-s}, 2^{-k} p+(m+1)2^{-k-s}),$$
and we define $I_{t}^n$ in an analogous manner. With this notation, the two dyadic children of a given interval $I_s^m$ are $I_{s+1}^{2m}$ and $I_{s+1}^{2m+1}.$ The following lemma contains the key iteration step in the sparse domination argument.
 
\begin{lem}\label{lem:IterationLemma}
Let $0\leq m < 2^s$, $0\leq n < 2^t$ be integers and set $N=s+t$. If $\beta = \{\beta_I\}_{I\in\D}$ is a sequence with $\|\beta\|_\infty \leq 1$ which is nonzero for finitely many indices, then 
$$
\left| \sum_{J \in \mathcal{G}(I)} \beta_J\langle f_1, h_{J_{s}^m} \rangle \langle h_{J_{t}^n}, f_2 \rangle \right| \lesssim \bigg( \langle f_1 \rangle_{I} \langle f_2 \rangle_{I} \mu(I) +  \sum_{\substack{S\in \mathcal{B}^s(I), \, T \in \mathcal{B}^t(I): \\ \dyad(S,T) \leq  N+2}} \langle f_1 \rangle_{S}  \langle f_2 \rangle_{T} \sqrt{m(S)} \sqrt{m(T)} \bigg).
$$ 
 
\end{lem}

\begin{proof}
  Fix $I \subseteq I_0.$ We begin with the usual splitting with $f_j \one_{I}=g_j+b_j$ with respect to the Calder\'{o}n-Zygmund decomposition applied on $I$:

\begin{align*}
\left| \sum_{J \in \mathcal{G}(I)} \beta_J\langle f_1, h_{J_{s}^m} \rangle \langle h_{J_{t}^n}, f_2 \rangle \right| & \leq \left| \sum_{J \in \mathcal{G}(I)} \beta_J\langle g_1, h_{J_{s}^m} \rangle \langle h_{J_{t}^n}, g_2 \rangle \right|+ \left| \sum_{J \in \mathcal{G}(I)} \beta_J\langle g_1, h_{J_{s}^m} \rangle \langle h_{J_{t}^n}, b_2 \rangle \right|\\
& + \left| \sum_{J \in \mathcal{G}(I)} \beta_J\langle b_1, h_{J_{s}^m} \rangle \langle h_{J_{t}^n}, g_2 \rangle \right|+ \left| \sum_{J \in \mathcal{G}(I)} \beta_J\langle b_1, h_{J_{s}^m} \rangle \langle h_{J_{t}^n}, b_2 \rangle \right|  \\
& := \mathrm{(I)} + \mathrm{(II)} + \mathrm{(III)} +\mathrm{(IV)}.
\end{align*}

We handle the four terms in order. We estimate, using the fact that the $h_J$ are orthonormal, the $\Ltwo$ control on the good functions, and the fact that $\|\beta\|_{\ell^\infty} \leq 1$:

\begin{align*}
(I) & \leq \sum_{J \in \mathcal{G}(I)}| \langle g_1, h_{J_m^s}\rangle| \cdot | \langle h_{J_n^t}, g_2 \rangle|\\
& \leq \left( \sum_{K \in \D}| \langle g_1, h_{K}\rangle|^2\right)^{1/2} \left( \sum_{L \in \D}| \langle h_L, g_2 \rangle|^2\right)^{1/2} \\
& \leq \|g_1\|_{\Ltwo(\mu)} \|g_2\|_{\Ltwo(\mu)}\\
& \lesssim \left( \langle f_1 \rangle_{I}\right)^{1/2} \left(\int_{I} f_1 \, d\mu\right)^{1/2} \left( \langle f_2 \rangle_{I}\right)^{1/2} \left(\int_{I} f_2 \, d\mu\right)^{1/2} = \langle f_1 \rangle_{I} \langle f_2 \rangle_{I} \mu(I).
\end{align*}
This proves the desired estimate for (I). For (II), we have $$ \mathrm{(II)}  \leq \sum_{J \in \mathcal{G}(I)} \sum_{k} |\langle g_1, h_{J_s^m} \rangle| | \langle h_{J_t^n},b_{2,k} \rangle|.$$ 
However, many of the terms in this summation are zero due to the cancellation/support of the $b_{2,k}$ and the constraint $J \in \mathcal{G}(I).$ Indeed, the cancellation on the $b_{2,k}$ implies $J_t^n \subsetneq I_k^{(2)}$, while the condition $J \in \mathcal{G}(I)$ together with the support of $b_{2,k}$ implies that either $I_k \subsetneq J$ or $J\subset I_k^b.$ But in the latter case, $b_{2,k}$ is a constant function on the support of $h_{J_t^n}$ (which is contained in $I_k^b$), so these terms may be eliminated as well by the cancellation of $h_{J_t^n}.$ Note that given any fixed $I_k$, there are finitely many intervals $J$ that can satisfy these constraints, and the number of intervals only depends on $t$, not $k.$ Indeed, we may write $I_k \subsetneq J \subsetneq I_k^{(2+t)}.$  Therefore, after eliminating these terms, the double summation is equal to: 
$$\sum_{k} \sum_{\substack{J \in \mathcal{G}(I):\\ I_k \subsetneq J \subsetneq I_k^{(2+t)}}}  |\langle g_1, h_{J_s^m} \rangle| | \langle h_{J_t^n},b_{2,k} \rangle|.$$ 
Fix $J \in \mathcal{G}(I)$ satisfying the additional two conditions in the summation. We will obtain a uniform estimate on the inner summation. On the one hand, we estimate 
\begin{align*}
| \langle g_1, h_{J_s^m} \rangle| & = \sqrt{m(J_s^m)} \left| \langle g_1 \rangle_{J_{s+1}^{2m}}- \langle g_1 \rangle_{J_{s+1}^{2m+1}} \right|\\
& \leq  \sqrt{m(J_s^m)} \left( \left| \langle g_{1} \rangle_{J_{s+1}^{2m}}- \langle g_{1}  \rangle_{J_s^m}\right|+ \left|\langle g_{1}  \rangle_{J_s^m}  - \langle g_1 \rangle_{J_{s+1}^{2m+1}}      \right| \right)  \\
& \leq 2 \sqrt{m(J_s^m)} \|g_1\|_{\BMO} \lesssim \sqrt{m(J_s^m)} \langle f_1 \rangle_{I},
\end{align*}
by Lemma \ref{CZD}, \eqref{item:3}. On the other hand, using the $L^\infty$ estimate of the Haar functions, the $\Lone$ control on the bad function, and the fact that $J_s^m$ and $J_t^n$ share the common ancestor $J$, we estimate:
\begin{align*} 
| \langle h_{J_t^n}, b_{2,k} \rangle| & \leq \frac{1}{\sqrt{m(J_t^n)}} \|b_{2,k}\|_{\Lone(\mu)}\\
& \lesssim \frac{1}{\sqrt{m(J_t^n)}}  \int_{I_k}|f_2| \, d \mu. \\
& \lesssim  \frac{1}{\sqrt{m(J_s^m)}}  \int_{I_k}|f_2| \, d \mu,
\end{align*}
where in the last line we used our assumption that $\mu$ is balanced applied at most $N$ times. Thus, in particular, for fixed $k$, we have a bound $$|\langle g_1, h_{J_s^m} \rangle| | \langle h_{J_t^n},b_{2,k} \rangle| \lesssim \langle f_1 \rangle_{I} \int_{I_k}|f_2| \, d \mu,$$ that is independent of $J.$ Putting these two estimates together and summing, we have 
\begin{align*}
\mathrm{(II)} & \leq  \sum_{k} \sum_{\substack{J \in \mathcal{G}(I):\\ I_k \subsetneq J \subsetneq I_k^{(2+t)}}}  |\langle g_1, h_{J_s^m} \rangle| | \langle h_{J_t^n},b_{2,k} \rangle|\\
& \lesssim  \sum_{k} \sum_{\substack{J \in \mathcal{G}(I):\\ I_k \subsetneq J \subsetneq I_k^{(2+t)}}} \langle f_1 \rangle_{I} \int_{I_k}|f_2| \, d \mu\\
& \lesssim \sum_k \langle f_1 \rangle_{I} \int_{I_k}|f_2| \, d \mu \lesssim \langle f_1 \rangle_{I} \langle f_2 \rangle_{I} \mu(I),
\end{align*}
where in the penultimate estimate the implicit constant depends only on $t$ and  in the last step we used the fact that the $I_k \in \mathcal{B}(I)$ are disjoint and contained in $I$ for all $k$. Term (III) is handled similarly and we omit the details, so we are left with (IV). We estimate
$$\mathrm{(IV)} \leq  \sum_{J \in \mathcal{G}(I)} |\langle b_1, h_{J_{s}^m} \rangle| \cdot |\langle h_{J_{t}^n}, b_2 \rangle| .$$
So, we need to handle terms of the form $|\langle b_{1,j}, h_{J_s^m} \rangle| \cdot |\langle h_{J_t^n}, b_{2,k}\rangle|$ where $b_{1,j}$ is supported on $I_j$, $b_{2,k}$ is supported on $I_k$, and both $I_j$ and $I_k$ are selected intervals. Again, there are some constraints on the interval $J$. In particular, we must have $I_j \subsetneq J \subsetneq I_j^{(2+s)}$ and $I_k \subsetneq J \subsetneq I_k^{(2+t)}$, as we determined previously. So for fixed $j,k$, there are finitely many intervals $J$ that contribute to the sum, where the number of terms only depends on $s,t$. Therefore, similar to before, it suffices to consider the sum 
$$  \sum_{j} \sum_{k} \sum_{\substack{J \in \mathcal{G}(I):\\ I_j \subsetneq J \subsetneq I_j^{(2+s)} \text{ and } I_k \subsetneq J \subsetneq I_k^{(2+t)}} } |\langle b_{1,j}, h_{J_{s}^m} \rangle| \cdot |\langle h_{J_{t}^n}, b_{2,k} \rangle|     $$
and show it is controlled by the right hand side in the statement. Let us fix $J, j,k$ as in the summation for the time being, and consider a term $|\langle b_{1,j}, h_{J_{s}^m} \rangle| \cdot |\langle h_{J_{t}^n}, b_{2,k} \rangle|.$ Note that either $J_s^m= \widehat{I}_j$ or $J_s^m \subset I_j.$ In the former case, we can control the inner product as follows:
\begin{align*}
|\langle b_{1,j}, h_{J_{s}^m} \rangle|  & = \sqrt{m(\widehat{I}_j)} \left| \langle b_{1,j} \rangle_{I_j}- \langle b_{1,j} \rangle_{I_j^b}\right|\\
& = \sqrt{m(\widehat{I}_j)}\left| \langle f_1 \rangle_{I_j}- \langle f_{1} \one_{I_j} \rangle_{\widehat{I}_j} + \langle f_{1} \one_{I_j} \rangle_{\widehat{I}_j}- \langle f_1 \one_{I_j} \rangle_{I_j^b}\right|\\
& \sim \sqrt{m(I_j)} \langle f_1 \rangle_{I_j}.
\end{align*}
If $J_t^n=\widehat{I}_k$ as well, then we obviously have the analogous estimate and therefore:
\begin{equation} |\langle b_{1,j}, h_{J_{s}^m} \rangle| \cdot |\langle h_{J_{t}^n}, b_{2,k} \rangle| \lesssim \langle f_1 \rangle_{I_j} \langle f_2 \rangle_{I_k} \sqrt{m(I_j)} \sqrt{m(I_k)} \label{BB1} \end{equation}
Moreover, notice that $\dyad(I_j,I_k) \leq N+2$ and by definition, $I_j, I_k \in \mathcal{B}(I)$ and are disjoint unless $I_j=I_k$. In the latter case when $J_s^m \subset I_j$, we see that
\begin{align}
|\langle b_{1,j}, h_{J_{s}^m} \rangle|  & = \sqrt{m(J_s^m)} \left| \langle b_{1,j} \rangle_{J_{s-}^{m}}- \langle b_{1,j} \rangle_{J_{s+}^{m}}\right| \nonumber\\
& =\sqrt{m(J_s^m)} \left| \langle f_1 \rangle_{J_{s-}^{m}}- \langle f_{1} \one_{I_j} \rangle_{\widehat{I}_j} + \langle f_{1} \one_{I_j} \rangle_{\widehat{I}_j}- \langle f_1 \rangle_{J_{s+}^{m}}\right| \nonumber\\
& \leq \sqrt{m(J_s^m)} \langle f_1 \rangle_{J_{s-}^{m}}+ \sqrt{m(J_{s}^m)} \langle f_1 \rangle_{J_{s+}^{m}} \label{BB2} .
\end{align}
It suffices to control $\sqrt{m(J_{s}^m)} \langle f_1 \rangle_{J_{s-}^{m}}$; the other term is obviously handled in the same way. If $J_{s-}^{m} \in \mathcal{B}(I_j)$, we leave the estimate as is. So suppose that $J_{s-}^{m} \notin \mathcal{B}(I_j).$ Then either $\langle f_1 \rangle_{J_{s-}^{m}} \leq 16 \langle f_1 \rangle_{I_j}$ or there exists $K \subsetneq I_j$ with $K \in \mathcal{B}(I_j)$ and $J_s^m \subsetneq K.$ In the former case, we have the estimate:
\begin{equation} \sqrt{m(J_s^m)} \langle f_1 \rangle_{J_{s-}^{m}} \lesssim  \sqrt{m(I_j)} \langle f_1 \rangle_{I_j} , \label{BB3}\end{equation}
using the balanced property of $\mu.$ In the latter case, we apply the same procedure, letting $K$ assume the role of $I_j.$ That is, either $J_{s-}^{m} \in \mathcal{B}(K),$ in which case we stop and leave the original estimate as is, or we have the control $\langle f_1 \rangle_{J_{s-}^{m}} \leq 16 \langle f_1 \rangle_{K}$, or there exists an interval $K' \in \mathcal{B}(K)$ so that $J_s^{m-} \subsetneq K',$ and we apply the procedure again to $K'.$ Notice that $\dyad(J_{s-}^{m},I_j)\leq s.$ Iterating this procedure at most $s$ times and using the balanced condition, we can find an interval $S \in \mathcal{B}^s(I)$ (more precisely, $\mathcal{B}^{s-1}(I_j)$) containing $J_{s-}^{m}$ and satisfying
\begin{equation} \sqrt{m(J_s^m)} \langle f_1 \rangle_{J_{s-}^{m}} \lesssim  \sqrt{m(S)} \langle f_1 \rangle_{S} , \label{BB4}\end{equation}
where the implicit constant depends only on $\mu$ and s. Moreover, the same reasoning applies to $J_t^n$, giving us an estimate
\begin{equation} \sqrt{m(J_{t}^n)} \langle f_1 \rangle_{J_{t-}^{n}} \lesssim  \sqrt{m(T)} \langle f_1 \rangle_{T} \label{BB5} .\end{equation}
for some $T \in \mathcal{B}^{t-1}(I_k)$ (or $T=I_k$) with $J_{t-}^{n} \subset T.$ Let $S,T$ be as above. We observe that $\dyad(S,T) \leq N+2$, since both have common ancestor $J.$ Examining estimates \eqref{BB1} through \eqref{BB5} and using the fact that the number of intervals $J$ satisfying these constraints only depends on $s,t$, we deduce that for fixed $j,k$, we have the bound:
\begin{align*}
& \sum_{\substack{J \in \mathcal{G}(I):\\ I_j \subsetneq J \subsetneq I_j^{(2+s)} \text{ and } I_k \subsetneq J \subsetneq I_k^{(2+t)}} } |\langle b_{1,j}, h_{J_{s}^m} \rangle| \cdot |\langle h_{J_{t}^n}, b_{2,k} \rangle| \\
& \lesssim \bigg(\langle f_1 \rangle_{I_j} \langle f_2 \rangle_{I_k} \sqrt{m(I_j)} \sqrt{m(I_k)}  +\sum_{\substack{S \in \mathcal{B}^s(I_j), \, T \in \mathcal{B}^t(I_k):\\ \dyad(S,T) \leq N+2}} \langle f_1 \rangle_{S}  \langle f_2 \rangle_{T} \sqrt{m(S)} \sqrt{m(T)}  \bigg),
\end{align*} where the implicit constant depends on $s,t$. Summing over $j,k$ (which correspond to disjoint intervals $I_j, I_k$) then completes the proof.
\end{proof}

We can now turn to the proof of Theorem \ref{th:ModifiedSparse}.

\begin{proof}[Proof of Theorem \ref{th:ModifiedSparse}] We first describe the construction of the sparse collection $\Ss$ and then show how Lemma \ref{lem:IterationLemma} leads to the desired domination. With the notation introduced in \eqref{eq:notationGB}, we assume that $f_1,f_2$ are supported on an interval $\tilde{I}_0\in \D$, we set $I_0=\tilde{I}_0^{\max\{s,t\}+1}$ and
$$
\Ss:= I_0 \cup \bigcup_{k=0}^\infty \mathcal{B}_k(I_0).
$$
$\Ss$ is independent of $s$ and $t$ and only depends on $f_1$ and $f_2.$ We need to show that it is sparse, and we do so showing that it satisfies \eqref{eq:CarlesonPacking}. By an inductive argument, it suffices to show that for a given $I \in \mathcal{B}_j(I_0),$ there holds $$\sum_{\substack{J \in \mathcal{B}_{j+1}(I_0):\\J \subset I}} \mu(J)  \leq \frac{\mu(I)}{2}.$$ For each such interval $J$, select $s(J) \in \{1,2\}$ such that $\langle f_{s(J)} \rangle_{J} \geq 16 \langle f_{s(J)} \rangle_{I}.$ We estimate, using the fact that the intervals in the sum are pairwise disjoint:
 \begin{align*}
\sum_{\substack{J \in \mathcal{B}_{j+1}(I_0):\\J \subset I}} \mu(J) & \leq \sum_{\substack{J \in \mathcal{B}_{j+1}(I_0):\\J \subset I}} \frac{\int_{J} f_{s(J)} \, d\mu}{16 \langle f_{s(J)} \rangle_{I}} \\
& \leq \frac{\mu(I)}{16} \sum_{\substack{J \in \mathcal{B}_{j+1}(I_0):\\J \subset I}}  \left( \frac{\int_{J} f_{1} \, d\mu}{\int_{I} f_{1} \, d\mu}+ \frac{\int_{J} f_{2} \, d\mu}{\int_{I} f_{2} \, d \mu} \right)\\
& \leq \frac{\mu(I)}{8}.
 \end{align*}
The claim then follows. We now turn to actually proving the domination. Since $f_1$ and $f_2$ are supported on $I_0$, we may assume that $\alpha_{J,K}^{I}=0$ if $I\cap I_0=\emptyset$. Therefore, we may decompose
\begin{align*}
    \langle Tf_1,f_2\rangle = \sum_{I \in \mathcal{O}(I_0)} \sum_{J \in \D_s(I)}\sum_{K \in \D_t(I)}\alpha_{J,K}^{I} \langle f_1, h_J \rangle \langle f_2, h_K \rangle \\
    + \sum_{m=0}^{2^s-1} \sum_{n=0}^{2^t-1} \sum_{I \in \D(I_0)} \alpha_{I_m^s,I_n^t}^{I} \langle f_1, h_{I_m^s} \rangle \langle f_2, h_{I_n^t} \rangle \\
    =: \langle T_0 f_1, f_2 \rangle + \sum_{m=0}^{2^s-1} \sum_{n=0}^{2^t-1} \langle T_{m,n} f_1, f_2 \rangle.
\end{align*}
Since we have the corona-type decomposition
$$
\D(I_0)= \bigcup_{J \in \Ss} \mathcal{G}(J),  
$$
we may apply Lemma \ref{lem:IterationLemma} to each $I\in \Ss$ and sum to get
$$
|\langle T_{m,n} f_1, f_2 \rangle| \lesssim \mathcal{A}_{\Ss}(f_1,f_2) +  \sum_{\substack{J,K \in \Ss: \\ 1 \leq \dyad(J,K) \leq  N+2}} \langle f_1 \rangle_{J}  \langle f_2 \rangle_{K} \sqrt{m(J)} \sqrt{m(K)}.
$$
after summing over $I\in\Ss$. Lemma \ref{lem:SparseReduction} then completes the estimate for the operators $T_{m,n}$ after passing to a second sparse collection $\Ss'$. Finally, because of our choice of $I_0$, if $I_0 \subsetneq I$ then $h_J$ is constant on any interval contained in $\tilde{I}_0$ if $J\in \D_s(I) \cup \D_t(I)$. Therefore,
\begin{align*}
   | \langle T_0 f_1, f_2 \rangle | & \leq \sum_{I \in \mathcal{O}(I_0)} \sum_{J \in \D_s(I)}\sum_{K \in \D_t(I)} |\langle f_1, h_J \rangle| |\langle f_2, h_K \rangle| \\
   & = \sum_{I \in \mathcal{O}(I_0)} \sum_{J \in \D_s(I)}\sum_{K \in \D_t(I)} |\langle \langle f_1\rangle_{\tilde{I}_0} \one_{\tilde{I}_0}, h_J \rangle| |\langle \langle f_2\rangle_{\tilde{I}_0} \one_{\tilde{I}_0}, h_K \rangle| \\
   & \lesssim_{s,t} \left( \sum_{J\in \D} |\langle \langle f_1\rangle_{\tilde{I}_0} \one_{\tilde{I}_0}, h_J \rangle|^2\right)^{\frac12} \left( \sum_{K\in \D} |\langle \langle f_2\rangle_{\tilde{I}_0} \one_{\tilde{I}_0}, h_K \rangle|^2\right)^{\frac12} \\
   & = \left\| \langle f_1\rangle_{\tilde{I}_0} \one_{\tilde{I}_0}\right\|_2 \left\| \langle f_2\rangle_{\tilde{I}_0} \one_{\tilde{I}_0}\right\|_2 = \langle f_1\rangle_{\tilde{I}_0} \langle f_2\rangle_{\tilde{I}_0} \mu(\tilde{I}_0), \\
\end{align*}
and the result follows.
\end{proof}

\begin{rem}\label{EliminateBrothers}
Notice that $\mathcal{C}_{\Ss}^0(f_1,f_2)\equiv 0$ for any pair $f_1,f_2$. Therefore, Theorem \ref{th:ModifiedSparse} with $(s,t)=(0,0)$ recovers the known result for Haar multipliers, which admit the usual sparse domination by $\mathcal{A}_{\Ss}(f_1,f_2)$ (see \cite{Lacey2017}). 
\end{rem}

\begin{rem}
Theorem \ref{th:thmA} is an immediate consequence of Theorem \ref{th:ModifiedSparse} when we take $N=1$. Note that this provides a sparse domination for the dyadic Hilbert transform $\Sha$ and -the same one- for its adjoint $\Sha^*$. In addition, by following the proof of Theorem \ref{th:ModifiedSparse}, one can obtain a refined sparse domination for $\Sha$ alone of the form
$$    | \langle \Sha f_1, f_2 \rangle|  \lesssim      C \left(\sum_{I \in \Ss} \langle f_1 \rangle_I \langle f_2 \rangle_{I} \mu(I)  + \sum_{\substack{I \in \Ss:\\ I^b_{-} \in \Ss}} \langle f_1 \rangle_{I} \langle f_2 \rangle_{I^b_{-}} m(I) + \sum_{\substack{I \in \Ss:\\ I^b_{+} \in \Ss}} \langle f_1 \rangle_{I} \langle f_2 \rangle_{I^b_{+}} m(I) \right)        $$
We leave the details to the interested reader. 
\end{rem}


\subsection{A maximal function} In what follows, given a pair of dyadic intervals $I,J$ and $p \in [1,\infty)$, let 
\begin{equation}   \Cpb(I,J)=  \begin{cases}            
1 & I=J  \\
\left( \frac{m(I)^{p/2} m(J)^{p/2}}{\mu(J) \mu(I)^{p-1}}\right) & \text{otherwise}.
\end{cases}  \label{ModApCorFactor} \end{equation}
The following modified maximal function that depends on $N\in\N$ is intimately related to sparse domination for $T\in \HS(s,t)$ for $s+t=N$.
\begin{definition}\label{MaximalCousins}
Given $N \in \mathbb{N}$, define the following maximal dyadic operator for $f \in \Lone(\mu)$: 

$$ \Maximal ^Nf(x):= \sup_{\substack{I, J \in \D:\\   \dyad(I,J)\leq N+2}}  \Coneb(I,J) \langle |f| \rangle_{I} \one_{J}(x) $$
\end{definition}
Under the assumption that $\mu$ is balanced, the operator $\Maximal ^N$ enjoys similar boundedness properties as the classical Hardy-Littlewood maximal function. 

\begin{lem}\label{ModMaxBounds} $\Maximal^N$ is bounded on $\Lp(\mu)$ for $1<p\leq \infty$ and is of weak-type $(1,1).$
\end{lem}

\begin{proof}
$\Maximal ^N$ is bounded on $L^\infty(\mu)$ with constant depending on the parameter $N$ and the balancing constants of $\mu$, so it suffices to show $\Maximal ^N$ is weak-type $(1,1)$ and then interpolate. Take $0\leq f \in \Lone(\mu)$, and for $\lambda>0$ let $$E_\lambda=\{x: \Maximal^N_{\D}f(x)>\lambda\}.$$The following countable collection of intervals then clearly covers $E_\lambda$:
$$\mathscr{C}_{\lambda}=\left\{J \in \D: \exists \, I \in \D \text{ so } \dyad(I,J)\leq N+2 \text{ and } \Coneb(I,J) \langle f\rangle_{I} > \lambda \right\}.$$
By selecting the maximal dyadic intervals in $\mathscr{C}_\lambda$, we obtain a countable, pairwise disjoint collection of intervals $\{J_j\}_{j=1}^{\infty}$ that covers $E_\lambda$, together with an associated collection of dyadic intervals $\{I_j\}_{j=1}^{\infty}$ satisfying 
\begin{equation} \Coneb(I_j,J_j) \langle f \rangle_{I_j}> \lambda \label{LargeAverage} \end{equation} 
for all $j.$ The intervals $\{I_j\}_{j=1}^{\infty}$ need not be pairwise disjoint. However, we can select a subsequence of maximal ones that are pairwise disjoint, which we denote by $\{I_{j_k}\}_{k=1}^{\infty}$. For fixed $k \in \mathbb{N}$, define the following index collections:
$$
A_k=\{j \in \mathbb{N}: I_j \subset I_{j_k}\}, \quad B_k=\{j \in A_k: J_j \cap I_{j_k} \neq \emptyset\}. 
$$
$\{A_k\}_{k=1}^{\infty}$ partitions $\mathbb{N}.$ First, we claim that
\begin{equation}| A_k \setminus B_k|\leq C_N. \label{CardBound} \end{equation}   To see this, notice that if $j \in A_k \setminus B_k$, then $I_j \subset I_{j_k}$ and $J_j \cap I_{j_k}= \emptyset.$ Notice that if $K$ is a common dyadic ancestor to $I_j$ and $J_j$, it must contain $I_{j_k}$, for otherwise it is contained in $I_{j_k}$, and thus $I_{j_k}$ contains $J_j$, a contradiction. Thus $K$ is also a common dyadic ancestor of $I_{j_k}$ and $J_j$. But then, by the definition of dyadic distance, we have $\dyad(J_j,I_{j_k}) \leq\dyad(J_j,I_{j}) \leq C_N$, which implies \eqref{CardBound}. Also, notice that for any $j \in A_k \setminus B_k$, by \eqref{LargeAverage} and the balanced hypothesis we have the uniform control 

\begin{align}
 \mu(J_j) & \leq \frac{1}{\lambda} \left(\frac{\mu(J_j)}{\mu(I_j)}\right) \left(\frac{\sqrt{m(I_{j}) m(J_{j})}}{\mu(J_{j})}\right) \int_{I_j} f \, d\mu \nonumber \\
 & \lesssim \frac{1}{\lambda} \int_{I_{j_k}} f \, d\mu \label{UniformWeakType} .
 \end{align}
On the other hand, if there exists $j \in B_k$ that satisfies $J_j \supset I_{j_k}$, then $J_j$ is the only element of $B_k$, since the $J_j$ are pairwise disjoint, and moreover $\mu(J_j)$ is controlled by \eqref{UniformWeakType} by the same argument as above. In this case, $|A_k|\leq C_N+1$ and we can redefine $B_k= \emptyset$ in the argument that follows. In the second case, we have $J_j \subsetneq I_{j_k}$ for all $j \in B_k$, but this time $B_k$ can be infinite. We have
\begin{equation} \bigcup_{j \in B_k} J_j \subset I_{j_k}, \label{WeakTypeContainment} \end{equation}
and moreover, by the definition of $I_{j_k}$ and using the fact that $m(I_{j_k}) \sim m(J_{j_k})$ by the balanced condition with implicit constant depending only on $N$:
\begin{align}
\mu(I_{j_k}) & \leq  \frac{1}{\lambda} \frac{\sqrt{m(I_{j_k}) m(J_{j_k})}}{\mu(J_{j_k})} \int_{I_{j_k}} f \, d\mu \nonumber\\
& \lesssim \frac{1}{\lambda} \int_{I_{j_k}} f \, d\mu \label{UniformWeak2}.    
\end{align}
Collecting these facts together, we then have
\begin{align*}
\mu(E_\lambda) & \leq \sum_{j=1}^{\infty} \mu(J_j)\\
& = \sum_{k=1}^{\infty} \sum_{j \in A_k \setminus B_k} \mu(J_j) +\sum_{k=1}^{\infty} \sum_{j \in B_k} \mu(J_{j}) \\
& := \mathrm{(I)} + \mathrm{(II)}. 
\end{align*}
Note that for $\mathrm{(I)}$, using \eqref{CardBound} and \eqref{UniformWeakType} and recalling that the $I_{j_k}$ are pairwise disjoint, we have the control 
\begin{align*}
\sum_{k=1}^{\infty} \sum_{j \in A_k \setminus B_k} \mu(J_j) & \lesssim \frac{1}{\lambda}  \sum_{k=1}^{\infty} \sum_{j \in A_k \setminus B_k} \int_{I_{j_k}} f \, d\mu\\
& \leq \frac{(C_N+1)}{\lambda} \sum_{k=1}^{\infty} \int_{I_{j_k}} f \, d\mu \\
& \lesssim \frac{\|f\|_{\Lone(\mu)}}{\lambda},
\end{align*}
which is the bound we want. We bound $\mathrm{(II)}$ using the pairwise disjointess of the $J_j$, \eqref{WeakTypeContainment}, and \eqref{UniformWeak2}: 
\begin{align*}
\sum_{k=1}^{\infty} \sum_{j \in B_k} \mu(J_{j}) & \leq \sum_{k=1}^{\infty} \mu(I_{j_k})\\
& \lesssim \frac{1}{\lambda} \sum_{k=1}^{\infty} \int_{I_{j_k}} f \, d\mu \\
& \lesssim \frac{\|f\|_{\Lone(\mu)}}{\lambda},
\end{align*}
which completes the proof. 
\end{proof}

$\Maximal ^N$ admits the same sparse domination as any $T\in\HS(s,t)$ for $N=s+t$. As before, we prove the sparse domination for the bi-sublinear form defined by it. 

 \begin{theorem}\label{MaximalSparse}
 Let $N \in \mathbb{N}.$ There exists $\eta \in (0,1)$ so that for each pair of compactly supported, bounded, nonnegative functions $f_1, f_2 \in \Lone(\mu)$, there exists an $\eta$-sparse collection $\Ss \subset \D$ so that 

 \begin{equation}
 \mathlarger{|} \left \langle \Maximal ^N f_1, f_2 \right \rangle  \mathlarger{|}  \lesssim (\mathcal{A}_{\Ss}(f_1,f_2)+ \mathcal{C}_{\Ss}^N(f_1,f_2)).
 \label{MaximalSparseBound} \end{equation}
\end{theorem}

 \begin{proof}
  Fix $f_1,f_2$ as in the statement. By Lemma \ref{lem:SparseReduction}, it suffices to show there exists $\eta \in (0,1)$ and an $\eta$-sparse collection $\Ss$ so that 

  \begin{equation}
 \mathlarger{|} \left \langle \Maximal ^N f_1, f_2 \right \rangle  \mathlarger{|}  \lesssim \mathcal{A}_{\Ss}(f_1,f_2)+ \sum_{\substack{J, K \in \Ss:\\ 1 \leq \dyad(J,K) \leq N+2}}\langle f_1 \rangle_{J} \langle f_2 \rangle_{K}\, \, \sqrt{m(J)} \sqrt{m(K)}.   
 \label{MaximalSparseRedu}\end{equation}
We will use a level set argument to prove \eqref{MaximalSparseRedu}. For $n \in \mathbb{Z}$ set 
$$
E_n=\{x \in \mathbb{R}: \Maximal ^Nf_1(x)>2^n\}.
$$
Cover $E_n \setminus E_{n+1}$ by maximal dyadic intervals $\{J_n^j\}_{j=1}^{\infty}$ so that for each $J_n^j$, there exists a corresponding $I_n^j \in \D$ satisfying $$ 2^n < \Coneb(I_n^j,J_n^j) \langle f_1 \rangle_{I_n^j} \leq 2^{n+1}, \quad \dyad(I_n^j,J_n^j) \leq N+2.$$
For fixed $n$, $\{J_n^j\}$ is a pairwise disjoint collection of intervals, although this is no longer necessarily the case when one varies the parameter $n$ and moves up and down level sets. However, even for fixed $n$, the collection $\{I_n^j\}$ need not be pairwise disjoint. Even so, by using an argument similar to that in Lemma \ref{ModMaxBounds}, in the bound we obtain we can replace $\{I_n^j\}$ with a subcollection of intervals which satisfy a finite overlap condition. We next describe the details of passing to this subcollection and obtaining the sparse bound. First, for fixed $n$, select a maximal subsequence $\{I_n^{j_k}\}_{k=1}^{\infty}$ among the $I_n^j.$ This collection will be pairwise disjoint. Then, for each $k$, define index collections $A_n^k$ and $B_n^k$ as in the proof of Lemma \ref{ModMaxBounds}. Recall we argued that for each $k$, we have $|A_n^{k} \setminus B_n^k|\leq C_N$, and this bound is independent of $n$ as well. If $|B_n^k|=1$, let $\ell(k) \in B_n^k$ denote the unique index so that $I_{n}^{\ell(k)} \subset I_{n}^{j_k}$ and $J_{n}^{\ell(k)} \cap I_{n}^{j_k} \neq \emptyset.$ Let $C_n$ denote the collection of indices $k$ so that $|B_n^k|=1$ and $J_{n}^{\ell(k)} \supset I_{n}^{j_k}.$ Set 
\begin{equation*} \Ss_n^1:= \{J_n^j\}_{j=1}^{\infty}, \quad \Ss_n^2:= \{I_n^{j_k}\}_{k=1}^{\infty}, \quad \Ss_n^3:=\bigcup_{k=1}^{\infty} \{I_n^j\}_{j \in A_n^k \setminus B_{n}^k}, \quad    \Ss_n^4:=  \{I_n^{\ell(k)}\}_{k \in  C_n} .
\end{equation*}
The collections $\Ss_n^1$, $\Ss_n^1$, and $\Ss_n^4$ are pairwise disjoint by construction, while the collection $\Ss_{n}^3$ has finite overlap. Next, we estimate:
\begin{align*}
 \mathlarger{|} \left \langle \Maximal ^N f_1, f_2 \right \rangle  \mathlarger{|}  & = \int_{\mathbb{R}} \Maximal ^N f_1 \cdot f_2 \, d \mu\\
 & \sim \sum_{n =-\infty}^{\infty} 2^n \int_{E_n \setminus E_{n+1}} f_2 \, d\mu\\
 & \leq  \sum_{n =-\infty}^{\infty} 2^n  \sum_{j=1}^{\infty}\int_{J_n^j} f_2 \, d\mu\\
 & = \sum_{n =-\infty}^{\infty} 2^n \sum_{k=1}^{\infty} \sum_{j \in A_n^k \setminus B_n^k}\int_{J_n^j} f_2 \, d\mu + \, \sum_{n =-\infty}^{\infty} 2^n \sum_{k=1}^{\infty} \sum_{j \in B_n^k}\int_{J_n^j} f_2 \, d\mu.
 \end{align*}
For the first sum, we have:
\begin{align*}
\sum_{n =-\infty}^{\infty} 2^n \sum_{k=1}^{\infty} \sum_{j \in A_n^k \setminus B_n^k}\int_{J_n^j} f_2 \, d\mu  & \lesssim \sum_{n =-\infty}^{\infty}  \sum_{k=1}^{\infty} \sum_{j \in A_n^k \setminus B_n^k} \Coneb(I_n^j,J_n^j) \langle f_1 \rangle_{I_n^j} \langle f_2 \rangle_{J_n^j} \mu(J_n^j)\\
& \lesssim \sum_{n =-\infty}^{\infty}  \sum_{\substack{I \in \Ss_n^3,\, J \in \Ss_n^1:\\ \operatorname{dist}_{\D(I,J) \leq N+2}}}  \langle f_1 \rangle_{I} \langle f_2 \rangle_{J} \sqrt{m(I)m(J)}.
\end{align*}
For the second sum, we first consider the case when $k \in C_n$. Then it is easy to see we arrive at the estimate:
$$ \sum_{n =-\infty}^{\infty} 2^n \sum_{k \in C_n} \int_{J_n^{\ell(k)}} f_2 \, d\mu \lesssim  \sum_{n =-\infty}^{\infty} \sum_{\substack{I \in \Ss_n^4, \, J \in \Ss_n^1:\\ \operatorname{dist}_{\D(I,J) \leq N+2}}}  \langle f_1 \rangle_{I} \langle f_2 \rangle_{J} \sqrt{m(I)m(J)}.   $$
If $k \not\in C_n$, then any $j' \in B_n^k$ must satisfy $J_n^{j'} \subset I_{n}^{j_k}.$ Then we have, using the pairwise disjointness of the $J_n^{j'}$:
\begin{align*}
\sum_{n =-\infty}^{\infty} 2^n \sum_{k \not \in C_n} \sum_{j' \in B_n^k} \int_{J_n^{j'}} f_2 \, d\mu
& \lesssim \sum_{n =-\infty}^{\infty} \sum_{k \not\in C_n} \Coneb(I_n^{j_k}, J_n^{j_k}) \langle f_1 \rangle_{I_n^{j_k}} \int_{I_n^{j_k} } f_2 \, d\mu\\
& \lesssim \sum_{n =-\infty}^{\infty}  \sum_{I \in \Ss_n^2}  \langle f_1 \rangle_{I} \langle f_2 \rangle_{I} \mu(I).
\end{align*}
Put $$\Ss^j= \bigcup_{n=-\infty}^{\infty} \Ss_n^j, \quad \mbox{ and } \quad \Ss= \bigcup_{j=1}^{4} \Ss^j.$$ Altogether, we have proven the desired estimate \eqref{MaximalSparseRedu} for the collection $\Ss.$ We claim that the collection $\Ss$ is sparse. For readability, we separate out this technical section of the proof into a dedicated lemma. The proof of Theorem \ref{MaximalSparse} is complete, up to the verification of the sparsity of $\Ss$.
\end{proof}

\begin{lem} There exists $\eta \in (0,1)$ so that the collections $\Ss_j$, constructed in the proof of Theorem \ref{MaximalSparse}, are individually $\eta$-sparse, for $j \in \{1,2,3, 4\}.$ Therefore, $\Ss$ is $\eta'$-sparse for some $0<\eta'<\eta$.
\end{lem} 
 
\begin{proof}
We first argue for the sparsity of $\Ss^1.$ It is enough to show that $\Ss^1$ is a union of finitely many sparse collections. To begin with, note that if $J_{n'}^{j'} \subset J_n^j$ for some indices $n,j,n',j'$, we necessarily have $n' \geq n$, since $J_n^j \subset E_n$ and $J_{n'}^{j'} \cap E_{n'+1}^c \neq \emptyset$ by construction. By a standard argument, it suffices to show that there exists a sufficiently large positive integer $m$ (depending only on $N$) so that for $n'=n+m$, we have, for each interval $J_n^j \in \Ss_n^1$,
\begin{equation} 
\sum_{j':\, J_{n'}^{j'} \subset J_n^j}\mu(J_{n'}^{j'}) \leq \frac{1}{2} \mu(J_n^j) \label{CPackingMaximal}.
\end{equation}
First, we split the left hand side of \eqref{CPackingMaximal} as follows:

\begin{align*}
& \sum_{\substack{j':\, J_{n'}^{j'} \subset J_n^j\\ I_{n'}^{j'} \not\subset J_n^j}}\mu(J_{n'}^{j'}) + \sum_{\substack{j':\, J_{n'}^{j'} \subset J_n^j\\ I_{n'}^{j'} \subset J_n^j}}\mu(J_{n'}^{j'})\\
& := \mathrm{(I)}+\mathrm{(II)}.
\end{align*}
To estimate $\mathrm{(I)}$, notice that if $I_{n'}^{j'} \not\subset J_n^j$, we have $\dyad(J_{n'}^{j'}, J_n^j) \leq \dyad(J_{n'}^{j'}, I_{n'}^{j'})\leq N+2$, so the summation actually has at most $C_N$ terms. Moreover, we also have $\dyad(I_{n'}^{j'}, J_n^j) \leq N+2$.  Then we estimate, using these facts, the balanced hypothesis, and the fact that $J_n^j$ intersects $E_{n+1}^c$:
\begin{align*}
\mathrm{(I)} & \leq \frac{1}{2^{n'}} \sum_{\substack{j':\, J_{n'}^{j'} \subset J_n^j\\ I_{n'}^{j'} \not\subset J_n^j}} \sqrt{m(I_{n'}^{j'}) m(J_{n'}^{j'})} \langle f_1 \rangle_{I_{n'}^{j'}} \\
& \lesssim \frac{1}{2^{n'}} \sum_{\substack{j':\, J_{n'}^{j'} \subset J_n^j\\ I_{n'}^{j'} \not\subset J_n^j}} \sqrt{m(I_{n'}^{j'}) m(J_{n}^{j})} \langle f_1 \rangle_{I_{n'}^{j'}}\\
& \leq \frac{1}{2^{m-1}} \sum_{\substack{j':\, J_{n'}^{j'} \subset J_n^j\\ I_{n'}^{j'} \not\subset J_n^j}} \mu(J_n^j) =  \frac{C_N}{2^{m-1}} \mu(J_n^j),
\end{align*}
so we clearly have the desired control on $\mathrm{(I)}$ assuming a sufficiently large choice of $m.$ To control $\mathrm{(II)}$, we begin by observing that by definition, 
$$
\langle f_1 \rangle_{J_n^j}= \Coneb(J_n^j,J_n^j) \langle f_1 \rangle_{J_n^j} \leq 2^{n+1}   , \quad \text{and} \quad \bigcup_{\substack{j':\, J_{n'}^{j'} \subset J_n^j\\ I_{n'}^{j'} \subset J_n^j}} J_{n'}^{j'}\subset \left\{x \in J_n^j: \Maximal ^N(\one_{J_n^j}f_1)(x) > 2^{n'} \right \}.
$$
Then we may estimate, using the pairwise disjointness of $J_{n'}^{j'}$ and the weak-type estimate in the proof of Lemma \ref{ModMaxBounds}:
\begin{align*}
\mathrm{(II)} & \leq \mu\left( \left\{x \in J_n^j: \Maximal ^N(\one_{J_n^j}f_1)(x) > 2^{n'} \right \}\right)\\
& \leq \frac{\|\Maximal ^N\|_{\Lone(\mu) \rightarrow \Lonew(\mu)}}{2^{n'}} \int_{J_n^j}f_1 \, d \mu \\
& \leq \frac{\|\Maximal ^N\|_{\Lone(\mu) \rightarrow \Lonew(\mu)}}{2^{m-1}} \mu(J_n^j),
\end{align*}
which also  establishes the necessary control for $\mathrm{(II)}$ as long as $m$ is chosen large enough. Thus, we conclude the collection $\Ss^1$ is sparse.

Next, is not difficult to show that if we prove the collection $\Ss^2$ is sparse, then this will also imply that collections $\Ss^3$ and $\Ss^4$ are sparse, so we focus on checking the Carleson packing condition for each $I=I_{n}^{j_k} \in \Ss^2.$ To begin with, choose $a$ to be the unique positive integer satisfying $$2^a<\langle f_1 \rangle_{I_n^{j_k}}= \Coneb(I_n^{j_k},I_{n}^{j_k})\langle f_1 \rangle_{I_n^{j_k}} \leq 2^{a+1}. $$ This implies that $I_{n}^{j_k} \subset E_a.$ Now, we split
\begin{equation} \sum_{\substack{K \in \Ss^2:\\ K \subsetneq I_n^{j_k}}} \mu(K)= \sum_{n'< a} \sum_{\substack{K \in \Ss_{n'}^2:\\ K \subsetneq I_n^{j_k}}} \mu(K)+ \sum_{n' \geq a} \sum_{\substack{K \in \Ss_{n'}^2:\\ K \subsetneq I_n^{j_k}}} \mu(K). \label{InteriorSplit} \end{equation}
Take $K=I_{n'}^{j_\ell'}$ as in the first sum. Note we must have $J_{n'}^{j_{\ell}'} \cap I_{n}^{j_k}= \emptyset.$ Indeed, $I_{n}^{j_k}$ is contained in $E_{a}$, while $J_{n'}^{j_{\ell}'}$ nontrivially intersects $E_{n'+1}^c$, which is disjoint from $E_a$ when $n'< a.$ Therefore, $\dyad(I_{n'}^{j_{\ell}'}, I_n^{j_k}) \leq \dyad(J_{n'}^{j_{\ell}'}, I_{n'}^{j_\ell'}) \leq N+2$, so the first sum in \eqref{InteriorSplit} is finite and in fact is bounded by $C_N \mu(I_{n}^{j_k}).$ The second summation  may be estimated as follows, using the pairwise disjointness of each collection $\Ss_{n'}^2$ and our choice of $a$:

\begin{align*}
\sum_{n' \geq a} \sum_{\substack{I_{n'}^{j_\ell'} \in \Ss_{n'}^2:\\ I_{n'}^{j_\ell'} \subsetneq I_n^{j_k}}} \mu(I_{n'}^{j_\ell'}) & \leq \sum_{n' \geq a}  \frac{1}{2^{n'}} \sum_{\substack{I_{n'}^{j_\ell'} \in \Ss_{n'}^2:\\ I_{n'}^{j_\ell'} \subsetneq I_n^{j_k}}} C_{1}(I_{n'}^{j_\ell'}, J_{n'}^{j_\ell'}) \int_{I_{n'}^{j_\ell'}} f_1 \, d \mu \\
&  \lesssim \sum_{n' \geq a}  \frac{1}{2^{n'}} \int_{I_n^{j_k}} f_1 \, d \mu  \\
& \leq  \sum_{n' \geq a}  \frac{1}{2^{n'}} 2^{a+1} \mu(I_{n}^{j_k}) \\
& \lesssim \mu(I_{n}^{j_k}).
\end{align*}
This bound proves the Carleson packing condition for $I$ and completes the proof. 
\end{proof}

\begin{rem} \label{SparseDiffComplex}
The sparse domination given in Theorem \ref{th:ModifiedSparse} depends on the complexity in an essential way. To prove this, we next show that a complexity $2$ Haar shift cannot be dominated by a complexity $1$ sparse form by a straightforward adaptation of Proposition \ref{prop:SparseFailure}. Define
$$ 
Tf:= \sum_{I \in \mathscr{D}} \langle f, h_I \rangle h_{I_{--}}, \quad f \in \Ltwo(\mu).
$$
Fix $0 < \eta < 1$ and let $\mu$ be as in Proposition \ref{prop:SparseFailure}. We will show there exist two sequences of compactly supported functions $\{f_j\},\{g_j\} \subset \Ltwo(\mu)$ such that for all $\eta$-sparse families $\Ss\subset \D$ 
\begin{equation}\label{sparsefailure2}
|\langle T f_j, g_j \rangle | \gtrsim j (|\A_\Ss(f_j,g_j)| +|\mathcal{C}_{\Ss}^1(f_j,g_j)| ).    
\end{equation}
 For $j \geq 2$, we take $f_j= \one_{I_{j-1}^b}$ and $g_j= \one_{I_{j+1}^b}.$ As a preliminary, we notice $f_j$ and $g_j$ have disjoint supports and if $I \in \D$ intersects both $\text{supp} (f_j)$ and $\text{supp} (g_j)$ non-trivially, we must have $I=I_k$ with $0 \leq k \leq j-2.$ We also observe

 \begin{equation} \Maximal ^1(f_j)(x) \lesssim \frac{\mu(I_{j-1}^b)}{\mu(I_{k-1})}, \,  \Maximal (g_j)(x) =\frac{\mu(I_{j+1}^b)}{\mu(I_{k-1})}, \, x \in I_{k}^b \text{ for }  0 < k \leq  j-3               ; \label{Maximal1PointwiseBound1} \end{equation}

  \begin{equation} \Maximal ^1(f_j)(x) \lesssim 1, \,  \Maximal (g_j)(x) = \frac{\mu(I_{j+1}^b)}{\mu(I_{k-1})}, \, x \in I_{k}^b  \text{ for }  j-2 \leq k \leq  j ; \label{Maximal1PointwiseBound2} \end{equation}

\begin{equation} \Maximal ^1(f_j)(x)\lesssim \frac{\mu(I_{j-1}^b)}{\mu(I_{j-2})}, \,  \Maximal (g_j)(x) = 1, \, x \in I_{j+1}^b ; \label{Maximal1PointwiseBound3} \end{equation}   

 \begin{equation}  \Maximal ^1(f_j)(x) \lesssim \frac{\mu(I_{j-1}^b)}{\mu(I_{j-2})}, \,  \Maximal (g_j)(x) =\frac{\mu(I_{j+1}^b)}{\mu(I_{j})}, \, x \in I_{k}^b \text{ for }  k > j+1  \label{Maximal1PointwiseBound4}  \end{equation}
 Using the sparsity of $\Ss$ and equations \eqref{Maximal1PointwiseBound1} through \eqref{Maximal1PointwiseBound4}, we can estimate the sums on the right hand side of \eqref{sparsefailure2} directly:

 \begin{align*}
 \sum_{I \in \Ss}  \langle f_j \rangle_{I}  \langle g_j \rangle_{I} \mu(I) & + \sum_{\substack{J, K \in \Ss:\\ \dyad(J,K) = 3,\\ J \cap K= \emptyset}} \langle f_j \rangle_{J} \langle g_j \rangle_{K}\, \, \sqrt{m(J)} \sqrt{m(K)} \\
 & \lesssim \sum_{I \in \Ss}  \langle f_j \rangle_{I}  \langle g_j \rangle_{I} \mu(E_I) + \sum_{\substack{J, K \in \Ss:\\ \dyad(J,K) = 3,\\ J \cap K= \emptyset}} \langle f_j \rangle_{J} \langle g_j \rangle_{K}\, \, \mu(E_K) \frac{\sqrt{m(J)} \sqrt{m(K)}}{\mu(K)} \\
& \lesssim  \int_{[0,1]} \Maximal ^{1} f_j \cdot  \Maximal g_j\, d\mu   \\
 & \leq \sum_{k=1}^{\infty} \int_{I_k^b} \Maximal ^1 f_j \cdot  \Maximal g_j \, d\mu \\
 & \lesssim \sum_{k=1}^{j-3} \left( \frac{\mu(I_{j-1}^b)}{\mu(I_{k-1})} \right) \cdot \left(\frac{\mu(I_{j+1}^b)}{\mu(I_{k-1})}\right)  \cdot   \mu(I_k^b) +    \sum_{k=j-2}^{j} \left(\frac{\mu(I_{j+1}^b)}{\mu(I_{k-1})}\right) \cdot \mu(I_k^b) \\
 & + \left(\frac{\mu(I_{j-1}^b)}{\mu(I_{j-2})}\right)  \cdot \mu(I_{j+1}^b) +  \sum_{k=j+2}^{\infty}\left(\frac{\mu(I_{j-1}^b)}{\mu(I_{j-2})}\right)  \cdot \left(\frac{\mu(I_{j+1}^b)}{\mu(I_{j})}\right)  \cdot\mu(I_k^b)  \\
     & \lesssim  \left(\frac{1}{j^4} \sum_{k=1}^{j-3} 1\right) + \frac{1}{j^3} + \frac{1}{j^3}+ \left(\frac{1}{j^2}\sum_{k=j+2}^{\infty} \frac{1}{k^2} \right) \lesssim \frac{1}{j^3}.
 \end{align*}
 Now we turn to a lower bound for $|\langle T f_j, g_j \rangle|.$ Given $I \in \D$, denote by $S_I$ the simple shift $S_I f= \langle f, h_I \rangle h_{I_{--}}.$ Then, by definition and using the support properties of $f_j$ and $g_j,$ we have $\langle S_I f_j, g_j \rangle= \langle f_j, h_I \rangle \cdot \langle h_{I_{--}}, g_j \rangle \neq 0 $ if and only if $I= I_k$ with $0 \leq k \leq j-2.$ Therefore, we may write
 $$
 \langle T f_j, g_j \rangle = \sum_{k=0}^{j-2} \langle S_{I_k} f_j, g_j \rangle = \sum_{k=0}^{j-2}  \langle f_j, h_{I_k} \rangle \cdot \langle h_{I_{k+2}}, g_j \rangle.
 $$
For the last term in the sum, computations similar to Proposition \ref{prop:SparseFailure} give
 $$\langle f_j, h_{I_{j-2}} \rangle \cdot \langle h_{I_{j}}, g_j \rangle= \sqrt{m(I_{j})m(I_{j-2}}) \sim  \frac{1}{j^2}.$$
 As before, all the other terms in the summation will similarly be positive, so \eqref{sparsefailure2} follows and the proof is complete. Of course, one can adapt this argument to higher complexities in an obvious way. The key point is that the higher complexity maximal function $\Maximal ^N$ only needs to be applied to one of the functions $f_j, g_j$, which leads to a tighter upper bound than if it were applied to both.
\end{rem}

\begin{rem} We can obtain a similar result as Remark \ref{SparseDiffComplex} for the sparse domination of the maximal function $M^N$ in Theorem \ref{MaximalSparse}. In particular, we can show that it is impossible to dominate $\langle M^{N} f, g \rangle$ by a complexity $N '$ sparse form where $N'<N$. The proof is very similar to the Haar shift case. 

\end{rem}


\subsection{$\Lp$ and weak-type estimates} We can now turn to estimates for modified sparse forms. We start with an $\Lp$ one.

\begin{lem}\label{lem:LpBounds}
Let $\Ss$ be sparse and $N\in\N$. For all nonnegative $f_1 \in \Lp(\mu)$ and $f_2 \in L^{p'}(\mu)$, there holds
$$
\mathcal{C}_{\Ss}^N(f_1,f_2) \lesssim_p \|f_1\|_{\Lp(\mu)} \|f_2\|_{L^{p'}(\mu)}.
$$
\end{lem}

\begin{proof}
Fix an integer $j$ satisfying $1 \leq j \leq N'.$ It clearly suffices to prove 
$$
\sum_{ \substack{ I \in \Ss: \\ c_j(I) \in \Ss}} \langle f_1 \rangle_{I} \langle f_2 \rangle_{c_j(I)} \sqrt{m(I) m(c_j(I))} \lesssim \|f_1\|_{\Lp(\mu)} \|f_2\|_{L^{p'}(\mu)},
$$
which we do as follows:
\begin{align*}
\sum_{ \substack{ I \in \Ss: \\ c_j(I) \in \Ss}} \langle f_1 \rangle_{I} \langle f_2 \rangle_{c_j(I)} \sqrt{m(I) m(c_j(I))} & \lesssim \sum_{ \substack{ I \in \Ss: \\ c_j(I) \in \Ss}} \langle f_1 \rangle_{I} \langle f_2 \rangle_{c_j(I)} m(I)^{1/p} m(c_j(I))^{1/p'} \\
& \lesssim \left( \sum_{I \in \Ss} (\langle f_1 \rangle_{I})^p \mu(I) \right)^{1/p} \left( \sum_{J \in \Ss} (\langle f_2 \rangle_{J})^{p'} \mu(J) \right)^{1/p'} \\
& \lesssim \left( \sum_{I \in \Ss} (\langle f_1 \rangle_{I})^p \mu(E_I) \right)^{1/p} \left( \sum_{J \in \Ss} (\langle f_2 \rangle_{J})^{p'} \mu(E_J) \right)^{1/p'}\\
& \leq \|\Maximal  f_1\|_{\Lp(\mu)}  \|\Maximal  f_2\|_{L^{p'}(\mu)}\\
& \leq   \| f_1\|_{\Lp(\mu)}  \| f_2\|_{L^{p'}(\mu)}.
\end{align*} 
where we have used the fact that $\mu$ is balanced in the first step. This finishes the proof. 
\end{proof}

\begin{prop}\label{prop:WeakType} Let $\Ss$ be sparse and $N\in\N$. Then
$$ \sup_{\substack{f_1: \\ \|f_1\|_{\Lone(\mu)} \leq 1}} \sup_{\substack{G \subset \mathbb{R}}} \inf_{\substack{G': \\ \mu(G) \leq 2 \mu(G')}} \sup_{f_2: |f_2| \leq \one_{G'}} \mathcal{C}_{\Ss}^N(|f_1|, |f_2|) < \infty.$$
\end{prop}

\begin{proof}
Let $G$ be an arbitrary Borel set, and $f_1$ with $\|f_1\|_{\Lone(\mu)}\leq 1$. We define the following set:
  $$H= \left\{x \in \mathbb{R}: \Maximal^N_{\D}f_1(x)> \frac{C_0}{\mu(G)}\right\}$$ 
 where we will choose the value of $C_0$ momentarily. By Lemma \ref{ModMaxBounds}, we get 
$$
\mu(H) \leq \|\Maximal^{N}_{\D}\|_{\Lone(\mu) \rightarrow \Lonew(\mu)} \frac{\mu(G)}{C_0} \leq \frac{\mu(G)}{2},
$$
after choosing $C_0 = 2 \|\Maximal^{N}_{\D}\|_{\Lone(\mu) \rightarrow \Lonew(\mu)}$.
Take $G'= G \setminus H.$ The above estimate shows $\mu(G) \leq 2 \mu(G')$, so $G'$ is a contender in the infimum. For any $f_2$ with $|f_2| \leq \one_{G'}.$, it is enough to prove 
$$
\sum_{\substack{I \in \Ss: \\ c_j(I) \in \Ss}} \langle f_1 \rangle_{I} \langle f_2 \rangle_{c_j(I)} \sqrt{m(I)} \sqrt{m(c_j(I))} \lesssim 1   
$$
for each $j \in \{1, \cdots, N'\}$. Let $I \in \Ss$ be such that $c_j(I) \in \Ss.$ If we have 
$$\frac{\sqrt{m(I) m(c_j(I))}}{\mu(c_j(I))} \langle f_1 \rangle_{I}> \frac{C_0}{\mu(G)},$$
then by definition $c_j(I) \subset H.$ But since $f_2$ is supported in $H^c,$ in this case we have $\langle f_2 \rangle_{c_j(I)}=0$.
Therefore, we may assume without loss of generality that if $I \in \Ss$ is such that $c_j(I) \in \Ss$, then we have the estimate
\begin{equation} \langle f_1 \rangle_{I} \leq \frac{\mu(c_j(I))}{\sqrt{m(I) m(c_j(I))}} \frac{C_0}{\mu(G)}. \label{AvgControl} \end{equation}
We then estimate, using Cauchy-Schwarz and the $\Ltwo$ bound for the ordinary maximal function:
\begin{align*}
\sum_{ \substack{ I \in \Ss: \\ c_j(I) \in \Ss}} \langle f_1 \rangle_{I} \langle f_2 \rangle_{c_j(I)} \sqrt{m(I) m(c_j(I))} & \lesssim \left(\sum_{ \substack{ I \in \Ss: \\ c_j(I) \in \Ss}} (\langle f_1 \rangle_{I})^2 m(I) \left(\frac{m(c_j(I))}{\mu(c_j(I))} \right) \right)^{1/2}   \left(\sum_{J \in \Ss} (\langle f_2 \rangle_{J})^2 \mu(J)  \right)^{1/2}\\
& \lesssim \left(\sum_{ \substack{ I \in \Ss: \\ c_j(I) \in \Ss}} (\langle f_1 \rangle_{I})^2 \left(\frac{m(I) m(c_j(I))}{\mu(c_j(I))^2}\right)  \mu(c_j(I)) \right)^{1/2} \left(\sum_{J \in \Ss} (\langle f_2 \rangle_{J})^2 \mu(E_J)  \right)^{1/2} \\
& \lesssim \left(\sum_{ \substack{ I \in \Ss: \\ c_j(I) \in \Ss}} (\langle f_1 \rangle_{I})^2 \left(\frac{m(I) m(c_j(I))}{\mu(c_j(I))^2}\right)  \mu(E_{c_j(I)})  \right)^{1/2} \mu(G)^{1/2}.
\end{align*}
We will now show $$\left(\sum_{ \substack{ I \in \Ss: \\ c_j(I) \in \Ss}} (\langle f_1 \rangle_{I})^2 \left(\frac{m(I) m(c_j(I))}{\mu(c_j(I))^2}\right)  \mu(E_{c_j(I)})  \right)^{1/2} \lesssim \mu(G)^{-1/2},$$
which is enough to conclude. Let $$\widetilde{f}= \sum_{ \substack{ I \in \Ss: \\ c_j(I) \in \Ss}} \Coneb(I,c_j(I)) \, \langle f_1 \rangle_{I} \one_{E_{c_j(I)}},$$ 
and notice that we have 
$$\|\widetilde{f}\|_{\Ltwo(\mu)} \sim    \left(\sum_{ \substack{ I \in \Ss: \\ c_j(I) \in \Ss}} (\langle f_1 \rangle_{I})^2 \left(\frac{m(I) m(c_j(I))}{\mu(c_j(I))^2}\right)  \mu(E_{c_j(I)})  \right)^{1/2}.
$$
Moreover, we have the $L^\infty$ estimate, using \eqref{AvgControl}:
$$
\| \widetilde{f}\|_{L^\infty(\mu)} \lesssim \frac{C_0}{\mu(G)}.
$$
Also notice that for almost every $x \in \mathbb{R}$, we have
$$  |\widetilde{f}(x)| \lesssim \Maximal ^N f_1(x).$$
Therefore, we can write using the distribution function, the weak-type estimate for $\Maximal ^N$ again, and the $\Lone$ normalization of $f_1$:  \begin{align*}\| \widetilde{f} \|_{\Ltwo(\mu )}^2 & = 2\int_{0}^\infty \lambda \mu(\{|\widetilde{f}|> \lambda\}) \, d \lambda \\
& \lesssim \int_{0}^{\frac{C}{\mu(G)}} \lambda \mu(\{ x: \Maximal^N_{\D} f_1(x)> \lambda     \}) \, d \lambda.\\
& \lesssim \|f_1\|_{\Lone(\mu)} \int_{0}^{\frac{C'}{\mu(G)}} 1 \, d \lambda\\
& \lesssim \frac{\|f_1\|_{\Lone(\mu)}}{\mu(G)}\leq \frac{1}{\mu(G)},
\end{align*}
proving the required estimate.
\end{proof}

\begin{rem}
As a consequence of Lemma \ref{lem:LpBounds}, Proposition \ref{prop:WeakType} and the corresponding ones for $\mathcal{A}_\Ss$, we recover the known results for the $\Lp$ bounds and weak-type estimates for $\Sha$ and other Haar shifts (see \cite{LMP2014}). 
\end{rem}

\begin{rem}
We leave as an interesting open question whether pointwise sparse estimates can be obtained, instead of the bilinear ones proved in this section.
\end{rem}

\section{Weighted Estimates} \label{sec:section3}

 We now consider weights $w$ on $\R$ with respect to a fixed balanced measure $\mu$. In what follows, we write 
 $$
 \langle f \rangle_{I}^w=\frac{1}{w(I)}\int_{I} f(x) w(x) \, d\mu(x),
 $$ 
 and $\Maximal^{w}$ to indicate the dyadic Hardy-Littlewood maximal operator with respect to the measure $w d\mu$. We now turn to  natural question of weighted $\Lp$ estimates for operators $T \in \HS(s,t)$. We first define the appropriate weight classes and quantities which reflect the complexity of the dyadic shift. 
 
\subsection{A new class of weights} For $p \in (1,\infty)$, we say $w \in \ApBalanced(\mu)$ if 
$$
[w]_{\ApBalanced(\mu)}:= \sup_{\substack{I, J \in \D:\\ J \in \{I, I_{-}^b, I_{+}^b\} \text{ or }\\ I \in \{J_{-}^b,J_{+}^b\} }} \Cpb(I,J) \langle w \rangle_{I} \langle w^{1-p'} \rangle_{J}^{p-1} < \infty. 
$$
If $p=1$, we say $w \in \AoneBalanced(\mu)$ if 
$$
[w]_{\AoneBalanced(\mu)}:=  \sup_{\substack{I, J \in \D:\\ J \in \{I, I_{-}^b, I_{+}^b\} \text{ or }\\ I \in \{J_{-}^b,J_{+}^b\}}} \Coneb(I,J)  \langle w \rangle_{I} \|w^{-1}\|_{L^\infty(J)}\leq C. 
$$
\begin{remark}
If $w \in \ApBalanced(\mu)$, $[w]_{A_p(\mu)} \leq [w]_{\ApBalanced(\mu)}.$ If $\mu$ is dyadically doubling, $A_p(\mu) = \ApBalanced(\mu)$. Proposition \ref{BadWeight} shows that, for general balanced measures, the inclusion $\ApBalanced(\mu) \subset A_p(\mu)$ is proper. As one would expect, we have the duality relationship $w \in \ApBalanced(\mu)$ if and only if $w^{1-p'} \in \ApprimeBalanced(\mu).$
\end{remark}

We are going to use a different quantitative characterization of the classes $\ApBalanced(\mu)$ that will be useful to study the boundedness of $\Maximal^N$. If $p \in (1,\infty)$ and $N \in \N$, we set 
$$
[w]_{A_p^N(\mu)}:= \sup_{\substack{I, J \in \D:\\ 0 \leq \dyad(J,K) \leq N+2}} \Cpb(I,J) \left(\langle w \rangle_{I} \langle w^{1-p'} \rangle_{J}^{p-1}\right). 
$$
For $p=1$, we set 
$$
[w]_{A_1^N(\mu)}:= \sup_{\substack{I, J \in \D:\\ \dyad(I,J)\leq N+2}} \Coneb(I,J) \langle w \rangle_{I} \|w^{-1}\|_{L^\infty(J)}.
$$
\begin{rem}
$[w]_{A_1^N(\mu)}<\infty$ if and only if there exists $C$ such that   
$$
\Maximal^N w(x) \leq C  w(x). 
$$
Moreover, $[w]_{A_1^N(\mu)}$ is equal to the infimum of all such $C$.
\end{rem}

\begin{lem} \label{lem:ApbvsApN}
Let $\mu$ be balanced. For $N \in\N$ and $1 \leq p< \infty$, 
$$
[w]_{\ApBalanced(\mu)} \lesssim [w]_{A_p^N(\mu)} \lesssim_N \left([w]_{\ApBalanced(\mu)}\right)^{2^{N-1}}.
$$
In particular, $ [w]_{A_p^N(\mu)} <\infty$ exactly when $[w]_{\ApBalanced(\mu)}<\infty$.
\end{lem}

\begin{proof}
We use induction on $N$. For the base case $N=1$, it is obvious that $[w]_{\ApBalanced} \lesssim [w]_{A_p^1(\mu)}.$
Conversely, if $w \in \ApBalanced(\mu)$ and $I \subset J$, for $p>1$ we have
\begin{align*}
\left( \frac{m(I)^{p/2} m(J)^{p/2}}{\mu(J) \mu(I)^{p-1}}\right) \langle w \rangle_{I} \langle w^{1-p'} \rangle_{J}^{p-1} & \leq \left( \frac{m(I)^{p/2} m(J)^{p/2}}{\mu(I)^{p}}\right) \langle w \rangle_{J} \langle w^{1-p'} \rangle_{J}^{p-1}\\
& \lesssim [w]_{\ApBalanced},
\end{align*}
 An entirely analogous argument works when $J \subset I$. The only other case to consider is when $J=I^b$, in which case we have the control:
$$  \left( \frac{m(I)^{p/2} m(I^b)^{p/2}}{\mu(I^b) \mu(I)^{p-1}}\right) \langle w \rangle_{I} \langle w^{1-p'} \rangle_{I^b}^{p-1} \lesssim       \langle w \rangle_{\widehat{I}} \langle w^{1-p'} \rangle_{\widehat{I}}^{p-1} .$$ The case $p=1$ can be dealt with with straightforward modifications. This completes the proof of the base case. 

Now, assume that the statement holds for some positive integer $j$. The inequality 
$$
[w]_{A_p^j(\mu)} \leq [w]_{A_p^{j+1}(\mu)}
$$ 
is immediate. Additionally, it is enough to consider disjoint intervals $I,J$ when bounding $[w]_{A_p^{j+1}(\mu)}$, so suppose $I,J$ are disjoint with $\dyad(I,J)=j+3.$ We may assume $\mu(I)< \frac{1}{2} \mu(\widehat{I})$ and $\mu(J)< \frac{1}{2} \mu(\widehat{J})$. For example, if $\mu(I)\geq \frac{1}{2} \mu(\widehat{I})$, then in fact $\mu(I) \sim \mu(\widehat{I})$ and we would have the control

\begin{equation}
\left( \frac{m(I)^{p/2} m(J)^{p/2}}{\mu(J) \mu(I)^{p-1}}\right) \langle w \rangle_{I} \langle w^{1-p'} \rangle_{J}^{p-1} \lesssim \left( \frac{m(\widehat{I})^{p/2} m(J)^{p/2}}{\mu(J) \mu(\widehat{I})^{p-1}}\right) \langle w \rangle_{\widehat{I}} \langle w^{1-p'} \rangle_{J}^{p-1} \leq [w]_{A_p^j(\mu)},  \label{DoublingReduct}  
\end{equation}
since $\dyad(\widehat{I},J)=j+2.$ Therefore, under this assumption we have $m(I) \sim \mu(I)$ and $m(J) \sim \mu(J)$, since $m(\widehat{I}) \sim \min \{\mu(I), \mu(I^b)\}$ for any dyadic interval $I$ and $\mu$ is balanced. Let $K$ denote the minimal common ancestor of $I$ and $J$. It is not difficult to see that either $I^{(2)} \subset K$ or $J^{(2)} \subset K$. Assume without loss of generality that $I^{(2)} \subset K.$ We consider two cases. If $\mu(\widehat{I}^b) \geq \frac{1}{2} \mu(I^{(2)})$, then we have
$$ \mu(\widehat{I}) \sim m(I^{(2)}) \sim m(I) \sim \mu(I),$$
and we can then use the doubling argument as in \eqref{DoublingReduct}. Otherwise, $\mu(\widehat{I}^b) \sim m(\widehat{I}^b)$ and we may estimate by H\"{o}lder's inequality:

\begin{align*}
\left( \frac{m(I)^{p/2} m(J)^{p/2}}{\mu(J) \mu(I)^{p-1}}\right) \langle w \rangle_{I} \langle w^{1-p'} \rangle_{J}^{p-1} & \leq \left( \frac{m(I)^{p/2} m(J)^{p/2}}{\mu(J) \mu(I)^{p-1}}\right) \langle w \rangle_{I} \langle w^{1-p'} \rangle_{\widehat{I}^b}^{p-1} \langle w \rangle_{\widehat{I}^b}\langle w^{1-p'} \rangle_{J}^{p-1} \\
& \sim \left[  \left( \frac{m(I)^{p/2} m(\widehat{I}^b)^{p/2}}{\mu(\widehat{I}^b) \mu(I)^{p-1}}\right) \langle w \rangle_{I} \langle w^{1-p'} \rangle_{\widehat{I}^b}^{p-1} \right]\\
& \times \left[  \left( \frac{m(\widehat{I}^b)^{p/2} m(J)^{p/2}}{\mu(J) \mu(\widehat{I}^b)^{p-1}}\right) \langle w \rangle_{\widehat{I}^b} \langle w^{1-p'} \rangle_{J}^{p-1} \right]\\
& \leq [w]_{A_p^j(\mu)}^2,
\end{align*}
since $\dyad(I, \widehat{I}^b)=3 \leq j+2$ and $\dyad(\widehat{I}^b,J)=j+2$. This completes the induction.
\end{proof}
    
\begin{rem}
Regardless of the sharpness of the inequalities in the statement of Lemma \ref{lem:ApbvsApN}, it is probably not the case that $[w]_{A_p^j(\mu)} \sim [w]_{A_p^k(\mu)}$ with implicit constants independent of $w,j,k$. More compelling evidence to support the fact that the characteristics are not quantitatively equivalent comes from the sparse domination in Theorems \ref{th:ModifiedSparse} and \ref{MaximalSparse}. The sparse domination depends in an essential way on the complexity. Remark \ref{SparseDiffComplex} shows that, in general, a complexity one sparse form $\mathcal{C}_{\Ss}^1$ cannot possibly dominate a dyadic shift with complexity $2$, and so on for higher complexities. Therefore, obtaining sharp quantitative weighted estimates for dyadic shift operators defined with respect to balanced measures remains an open problem, but one would expect the quantitative bounds to depend on the complexity of the operator. 
\end{rem}

The proof of our next result justifies the introduction of the quantities $[w]_{A_p^N(\mu)}$.

\begin{prop}\label{MaximalChar} Let $p \in (1, \infty)$ and $N \in \N.$ The operator $\Maximal^N$ is bounded on $\Lp(w d\mu)$ if and only if $w \in \ApBalanced(\mu)$. $\Maximal^N$ is bounded from $\Lone(wd \mu)$ to $\Lonew(w d\mu)$ if and only if $w \in \AoneBalanced(\mu).$ 
\end{prop}

\begin{proof}
We apply Lemma \ref{lem:ApbvsApN} to use the quantity $[w]_{A_p^N(\mu)}$ instead of $[w]_{\ApBalanced(\mu)}$. To begin with, fix $1<p<\infty$ and suppose that $\|\Maximal^N\|_{\Lp(w d\mu) \rightarrow \Lp(wd\mu)}=C.$ Let $I \in \D$ be arbitrary. Suppose that $w^{1-p'}$ is locally integrable (otherwise, consider $(w+\varepsilon)^{1-p'}$ and use a limiting argument). Let $J$ be any interval satisfying $0 \leq \operatorname{dist}(I,J)\leq N+2$ and take $f= w^{1-p'} \one_J.$ A computation shows that
$$
\|f\|_{\Lp(wd \mu)}= \left( \int_{J} w^{1-p'} \, d \mu\right)^{1/p}.
$$
Using the above and our hypothesis  we estimate:
\begin{align*}
C^p \left( \int_{J} w^{1-p'} \, d \mu \right)& \geq \|\Maximal^Nf\|_{\Lp(w d\mu)}^p\\
& \geq [\Coneb(J,I)]^p \int_{I} \left( \langle w^{1-p'} \rangle_J\right)^p w d\mu \\
& = \Cpb(I,J) \frac{\mu(J)}{\mu(I)} \int_{I} \left( \langle w^{1-p'} \rangle_J\right)^p w d\mu.
\end{align*}
Simple rearrangement then yields 
$$  
\Cpb(I,J) \langle w \rangle_{I} \langle w^{1-p'} \rangle_{J}^{p-1} \leq C^p,     
$$
which means $[w]_{A_p^N(\mu)} < \infty$, since $I,J$ were arbitrary. We prove the necessity when $p=1$ with an argument similar to \cite{M1972}. Put now
$$
\|\Maximal^N\|_{\Lone(w d\mu) \rightarrow \Lonew(w d\mu)} =C,
$$ 
and let $I, J \in \D$ satisfy $\dyad(I,J) \leq N+2.$ Let $\varepsilon>0$ be arbitrary and let $$A=\{x \in J: w(x)< \operatorname{ess \, inf}_{J}w+\varepsilon\}.$$
Note that $\mu(A)>0.$ Take $f=\one_{A}.$ Then, by definition we have
$$ I \subset \left\{x: \Maximal^Nf(x) \geq \frac{\mu(A)}{\mu(J)} \Coneb(J,I) \right \}.$$
It then follows from the supposed weak-type estimate that
\begin{align*}
w(I) & \leq C \frac{\mu(J)}{\Coneb(J,I)\mu(A)} \int_{A} w(x) d \mu(x)\\
& \leq C \frac{\mu(J)}{\Coneb(J,I)} (\operatorname{ess \, inf}_{J}w+\varepsilon).
\end{align*}
Since $\varepsilon$ was arbitrary and $C$ is independent of $\varepsilon$, simple rearrangment yields
$$ \Coneb(I,J) \langle w \rangle_{I} \leq C  (\operatorname{ess \, inf}_{J}w),    $$
which is precisely what we wanted to show. 

We next turn to the sufficiency of the condition $[w]_{A_p^N(\mu)} < \infty$. For $p>1$, it is immediately given by a weighted bound for the sparse form, which we postpone to the proof of Corollary \ref{WeightedBounds}, which is obtained independently from the maximal function, together with Theorem \ref{MaximalSparse}. When $p=1$, we argue like in Lemma \ref{ModMaxBounds}. Suppose $w \in \AoneBalanced(\mu)$ and take $f \in \Lone(w d \mu).$ Put $E_\lambda=\{x: \Maximal^Nf(x)>\lambda\}$, cover $E_\lambda$ by disjoint maximal intervals $\{J_j\}$ with corresponding averaging intervals $\{I_j\}$. Select the maximal dyadic intervals from $\{I_j\}$, label them $\{I_{j_k}\}$, and define the index sets $A_k$, $B_k$ as in the proof of Lemma \ref{ModMaxBounds}. We may write
 \begin{align*}
 w(E_\lambda) & \leq \sum_{j=1}^{\infty} \langle w \rangle_{J_j} \mu(J_j)\\
 & = \sum_{k=1}^{\infty} \sum_{j \in A_k \setminus B_k} \langle w \rangle_{J_j} \mu(J_j) +\sum_{k=1}^{\infty} \sum_{j \in B_k} \langle w \rangle_{J_j} \mu(J_{j}) \\
 & := \mathrm{(I)} + \mathrm{(II)}. 
\end{align*}
We control $\mathrm{(I)}$ as follows, noticing that $\Coneb(J_j,I_j) \langle w \rangle_{J_j} \leq \Maximal^Nw(x) \one_{I_j}(x)$ if $x \in I_j$:
\begin{align*}
\sum_{k=1}^{\infty} \sum_{j \in A_k \setminus B_k} \langle w \rangle_{J_j} \mu(J_j) & \lesssim \frac{1}{\lambda} \sum_{k=1}^{\infty} \sum_{j \in A_k \setminus B_k}  \langle w \rangle_{J_j} \Coneb(J_j,I_j)\int_{I_j} f \, d\mu\\
& \leq\frac{1}{\lambda }\sum_{k=1}^{\infty} \sum_{j \in A_k \setminus B_k} \int_{I_j} f \cdot \Maximal^N w \, d\mu\\
& \leq \frac{(C_N+1)[w]_{A_1^N(\mu)}}{\lambda} \sum_{k=1}^{\infty} \int_{I_{j_k}} f \, wd\mu \\
& \leq \frac{(C_N+1)[w]_{A_1^N(\mu)}}{\lambda} \|f\|_{\Lone(w d\mu)}.
\end{align*}
For $\mathrm{(II)}$, we can assume without loss of generality $J_j \subsetneq I_{j_k}$ for all $j \in B_k$. We then have
\begin{align*}
\sum_{k=1}^{\infty} \sum_{j \in B_k} w(J_{j}) & \leq  \sum_{k=1}^{\infty} \langle w \rangle_{I_{j_k}} \mu (I_{j_k})\\
& \lesssim \frac{1}{\lambda} \sum_{k=1}^{\infty} \langle w \rangle_{I_{j_k}}\int_{I_{j_k}} f \, d\mu \\
& \leq \frac{1}{\lambda} \sum_{k=1}^{\infty} \int_{I_{j_k}} f (\Maximal w)\, d\mu \\
& \leq \frac{[w]_{A_1^N(\mu)}}{\lambda} \sum_{k=1}^{\infty} \int_{I_{j_k}} f w d\mu \\
& \lesssim \frac{[w]_{A_1^N(\mu)}}{\lambda} \|f\|_{\Lone(w d\mu)}.
\end{align*}
\end{proof}    

\subsection{Haar Shifts} We now turn to weighted estimates for Haar shift operators. For convenience in what follows, we prove all weighted estimates for $p=2.$ However, all of the proofs carry over to the other values of $p$ with minor modifications. We first record the following standard lemma, which is an analog of an $A_\infty$ property we will need in the sequel. 

\begin{lem} \label{FairDivision} Let $w \in A_p(\mu).$ Then for any $I \in \D$ and $E_I \subset I$ with $\mu(E_I) \geq \eta (\mu(I)) $ for some $\eta \in (0,1)$, 
$$
w(E_I) \geq \frac{\eta^p}{[w]_{A_p(\mu)}} w(I).
$$
\end{lem}


\begin{prop} \label{prop:estimateCNS}
Let $\Ss$ be sparse and $N\in \N$. Then
$$
\|\C_\Ss^N\|_{\Lp(w d\mu) \rightarrow \Lp(w d\mu)} \lesssim_{N,p} [w]_{A_p(\mu)}^{\frac{(p-1)^2+1}{p(p-1)}}[w]_{A_p^{N}(\mu)}^{1/p}.   
$$
\end{prop}

\begin{proof} 
We only prove the result for $p=2$, and we proceed by duality: we need to show that for all bounded, compactly supported nonnegative $f_1 , f_2 \in \Ltwo(w d\mu)$, 
$$|\C_\Ss^N f_1, w f_2 \rangle| \lesssim [w]_{A_2(\mu)}[w]_{A_2^{N}(\mu)}^{1/2} \|f_1\|_{\Ltwo(w d\mu)}\|f_2\|_{\Ltwo(w d\mu)}.
$$ 
As in the proof of Theorem \ref{lem:LpBounds}, we only need to show the following estimate:
$$
\sum_{ \substack{ I \in \Ss: \\ c_j(I) \in \Ss}} \langle f_1 \rangle_{I} \langle w f_2 \rangle_{c_j(I)} \sqrt{m(I) m(c_j(I))}     \leq C  [w]_{A_2(\mu)} [w]_{A_2^{N}(\mu)}^{1/2} \|f_1\|_{\Ltwo (w)} \|f_2\|_{L^{2}(w)}.    
$$
We estimate as follows:
\begin{align*}
 \sum_{ \substack{ I \in \Ss: \\ c_j(I) \in \Ss}} \langle f_1 \rangle_{I} &\langle w f_2 \rangle_{c_j(I)} \sqrt{m(I) m(c_j(I))}
 = \sum_{ \substack{ I \in \Ss: \\ c_j(I) \in \Ss}} \langle w f_1 \rangle_{I}^{w^{-1}} \langle f_2 \rangle_{c_j(I)}^{w} \langle w^{-1} \rangle_{I}  \langle w \rangle_{c_j(I)}  \sqrt{m(I) m(c_j(I))} \\
& \sim  \sum_{ \substack{ I \in \Ss: \\ c_j(I) \in \Ss}} \langle w f_1 \rangle_{I}^{w^{-1}} \langle f_2 \rangle_{c_j(I)}^{w} \sqrt{\frac{m(I) m(c_j(I))}{\mu(I)\mu(c_j(I))}} \left(\langle w^{-1} \rangle_{I}  \langle w \rangle_{c_j(I)}\right)^{1/2} (w^{-1}(I))^{1/2} w(c_j(I))^{1/2}  \\
& \lesssim [w]_{A_2^{N}(\mu)}^{1/2} \sum_{ \substack{ I \in \Ss: \\ c_j(I) \in \Ss}} \langle w f_1 \rangle_{I}^{w^{-1}} \langle f_2 \rangle_{c_j(I)}^{w} (w^{-1}(I))^{1/2} w(c_j(I))^{1/2} \\
& \lesssim [w]_{A_2^{N}(\mu)}^{1/2} \left( \sum_{  I \in \Ss} (\langle w f_1 \rangle_{I}^{w^{-1}})^2 w^{-1}(I) \right)^{1/2} \left(\sum_{ J \in \Ss} (\langle f_2 \rangle_{J}^{w})^2  w(J)  \right)^{1/2}\\
& \lesssim [w]_{A_2(\mu)}[w]_{A_2^{N}(\mu)}^{1/2} \left( \sum_{  I \in \Ss} (\langle w f_1 \rangle_{I}^{w^{-1}})^2 w^{-1}(E_I)  \right)^{1/2} \left(\sum_{ J \in \Ss} (\langle f_2 \rangle_{J}^{w})^2  w(E_J) \right)^{1/2}\\
& \leq [w]_{A_2(\mu)} [w]_{A_2^{N}(\mu)}^{1/2} \|\Maximal^{w^{-1}}(f_1 w)\|_{\Ltwo(w^{-1} d \mu)} \|\Maximal^w(f_2 )\|_{\Ltwo(w d \mu)} \\
& \lesssim  [w]_{A_2(\mu)} [w]_{A_2^{N}(\mu)}^{1/2} \|f_1 w\|_{\Ltwo(w^{-1} d \mu)} \|f_2 \|_{\Ltwo(w d \mu)} \\
& = [w]_{A_2(\mu)} [w]_{A_2^{N}(\mu)}^{1/2} \|f_1\|_{\Ltwo(w d \mu)} \|f_2 \|_{\Ltwo(w d \mu)},\\
\end{align*}
where we used Lemma \ref{FairDivision} in the fifth line. This completes the proof. 
\end{proof}
 
 We can also obtain a weighted weak-type estimate for the sparse form when $p=1.$

\begin{prop} \label{WeightedWeakTypeSparse} Let $\Ss$ be sparse and $N \in \N$. Then
 $$ \|\C_\Ss^N \|_{\Lone(w d\mu) \rightarrow \Lonew(w d \mu) } \lesssim_N [w]_{A_1^N(\mu)}^{2}.
 $$
\end{prop}
\begin{proof}
It is enough to show
$$ 
\sup_{\substack{f_1: \\ \|f_1\|_{\Lone(w d \mu)} \leq 1}} \sup_{\substack{G \subset \mathbb{R}}} \inf_{\substack{G': \\ w(G) \leq 2 w(G')}} \sup_{f_2: |f_2| \leq \one_{G'}} \mathcal{C}_{\Ss}^N(|f_1|, w|f_2|) \lesssim [w]_{A_1^N(\mu)}^{2}.
$$
We follow the general scheme of the proof of Theorem \ref{prop:WeakType}. Let $G$ be an arbitrary Borel set, and $f_1$ with $\|f_1\|_{\Lone(w d \mu)}\leq 1$. Define
$$
H= \left\{x \in \mathbb{R}: \Maximal^N f_1(x)> \frac{C_0}{w(G)}\right\},
$$ 
 where $C_0$ is some sufficiently large positive constant. By Proposition \ref{MaximalChar}, $w(H) \leq w(G)/2$ as long as $C_0 \geq C_0' [w]_{A_1^N(\mu)}$, where $C_0'$ depends only on $\mu$ and $N$. As before, we take $G'= G \setminus H.$ Let $f_2$ be arbitrary with $|f_2| \leq \one_{G'}.$  It is enough to show, for $1 \leq j \leq C_N$, that we have
 $$
 \sum_{\substack{I \in \Ss: \\ c_j(I) \in \Ss}} \langle f_1 \rangle_{I} \langle w f_2 \rangle_{c_j(I)} \sqrt{m(I)} \sqrt{m(c_j(I))} \lesssim [w]_{A_1^N(\mu)}^{2} 
 $$
By the same argument as the proof of Proposition \ref{MaximalChar}, we may assume without loss of generality that if $I \in \Ss$ is such that $c_j(I) \in \Ss$, then we have the estimate
 $$
 \langle f_1 \rangle_{I} \leq \frac{\mu(c_j(I))}{\sqrt{m(I) m(c_j(I))}} \frac{C_0' [w]_{A_1^N(\mu)}}{w(G)}. 
 $$
 We then estimate, using Cauchy-Schwartz inequality, Lemma \ref{FairDivision}, the fact that $[w]_{A_2(\mu)} \leq [w]_{A_1^N(\mu)}$, and the weighted $\Ltwo$ bound for $\Maximal^w$:
 \begin{align*}
\sum_{ \substack{ I \in \Ss: \\ c_j(I) \in \Ss}} \langle f_1 \rangle_{I} \langle w f_2 \rangle_{c_j(I)} &\sqrt{m(I) m(c_j(I))} \\
&\lesssim \left(\sum_{ \substack{ I \in \Ss: \\ c_j(I) \in \Ss}} (\langle f_1 \rangle_{I})^2 \left(\frac{m(I) m(c_j(I))}{\mu(c_j(I))^2}\right) w(c_j(I)) \right)^{1/2}   \left(\sum_{J \in \Ss} (\langle wf_2 \rangle_{J})^2 \frac{\mu(J)^2}{w(J)} \right)^{1/2}\\
 & \lesssim [w]_{A_1^N(\mu)} \left(\sum_{ \substack{ I \in \Ss: \\ c_j(I) \in \Ss}} (\langle f_1 \rangle_{I})^2 \left(\frac{m(I) m(c_j(I))}{\mu(c_j(I))^2}\right)  w(E_{c_j(I)}) \right)^{1/2} \left(\sum_{J \in \Ss} (\langle f_2 \rangle_{J}^w)^2 w(E_J)  \right)^{1/2} \\
 & \lesssim [w]_{A_1^N(\mu)} \left(\sum_{ \substack{ I \in \Ss: \\ c_j(I) \in \Ss}} (\langle f_1 \rangle_{I})^2 \left(\frac{m(I) m(c_j(I))}{\mu(c_j(I))^2}\right)  w(E_{c_j(I)})  \right)^{1/2} w(G)^{1/2}.
\end{align*}
We will now show $$(\ast):=\left(\sum_{ \substack{ I \in \Ss: \\ c_j(I) \in \Ss}} (\langle f_1 \rangle_{I})^2 \left(\frac{m(I) m(c_j(I))}{\mu(c_j(I))^2}\right)  w(E_{c_j(I)})  \right)^{1/2} \lesssim [w]_{A_1^N(\mu)} w(G)^{-1/2}.$$ 
To see this, let 
$$\widetilde{f}= \sum_{ \substack{ I \in \Ss: \\ c_j(I) \in \Ss}} \Coneb(I,c_j(I)) \, \langle f_1 \rangle_{I} \one_{E_{c_j(I)}}, 
$$
which satisfies 
$$
 \|\widetilde{f}\|_{L^\infty(\mu)} \lesssim \frac{C_0}{w(G)}, \quad |\widetilde{f}| \lesssim \Maximal^N f_1, \quad \mbox{and} \quad \|\widetilde{f}\|_{\Ltwo(w d\mu)} \sim (\ast).$$
This, combined with the weighted weak-type estimate for $\Maximal^N$, and the $\Lone(w d \mu)$ normalization of $f_1$, yields  
\begin{align*}\| \widetilde{f} \|_{\Ltwo(w  d \mu )}^2 & = 2\int_{0}^\infty \lambda w(\{|\widetilde{f}|> \lambda\}) \, d \lambda \\
& \lesssim  \int_{0}^{\frac{C[w]_{A_1^N(\mu)}}{w(G)}} \lambda w (\{ x: \Maximal^N f_1(x)> \lambda\}) \, d \lambda.\\
& \lesssim [w]_{A_1^N(\mu)} \|f_1\|_{\Lone(w d \mu)} \int_{0}^{\frac{C[w]_{A_1^N(\mu)}}{w(G)}} 1 \, d \lambda\\
 & \lesssim \frac{[w]_{A_1^N(\mu)}^2}{w(G)},
\end{align*}
proving the required estimate.
\end{proof}

Using Theorem \ref{th:ModifiedSparse} and the estimates for $\mathcal{A}_\Ss$ (which are analogous to those in the case of the Lebesgue measure), we get the following consequence:

\begin{cor} \label{WeightedBounds} Let $N\in \N$. If $T\in\HS(s,t)$ and $s+t \leq N$ or $T=\Maximal^N$, then 
\begin{itemize}
\item If $p>1$,
$$ 
\|T\|_{\Lp(w d\mu) \rightarrow \Lp(w d\mu)} \lesssim [w]_{A_p(\mu)}^{\max\{1, \frac{1}{p-1}\}}+[w]_{A_p(\mu)}^{\frac{(p-1)^2+1}{p(p-1)}}[w]_{A_p^{N}(\mu)}^{1/p}. 
$$
\item The following weak-type estimate holds:
 $$ \|T\|_{\Lone(w d\mu) \rightarrow \Lonew(w d \mu) } \lesssim [w]_{A_1^N(\mu)}^{2}.$$
\end{itemize}
\end{cor}  

\begin{remark}\label{OneSidedWeight}
One can follow similar arguments to obtain ``one sided'' weight conditions that, for example, are sufficient for the boundedness on $\Ltwo(w d\mu)$ (or $\Lp(w d\mu)$) of $\Sha$ or other particular operators $T\in\HS(s,t)$. Indeed, one can check that the condition
\begin{equation} [w]_{A_2^{0,1}(\mu)}:=  \sup_{\substack{J \in \D:\\ K \in \{ J, J^{b}_{-}, J^{b}_{+}\}}}\Ctwob(K,J) \langle w \rangle_{K} \langle w^{-1} \rangle_{J} < \infty \label{HilbertWeight} \end{equation}
is sufficient for the boundedness of $\Sha$. However, such conditions are less natural in the sense that they are asymmetric with respect to $w$ and $w^{-1}$ and are specialized to the form of the particular operator under consideration. 
\end{remark}

We finally consider the necessity of the $\ApBalanced(\mu)$ condition. 

\begin{lem}  \label{PairBoundBelow}
Let $N \in \N$ and $p \in (1,\infty)$. Suppose there exists a constant $C$ so that
$$ \max_{\substack{s,t \in \N: \\ s+t \leq N }} \sup_{T\in\HS(s,t)} \|T \|_{\Lp(w d\mu) \rightarrow \Lp(w d\mu) } \leq C.
$$ 
Let $J,K \in \D$ be such that $2<\dyad(J,K)\leq N+2$ and $J \cap K= \emptyset.$
 Then  
 $$  \Cpb(J,K)) \left(\langle w \rangle_{J} \langle w^{1-p'} \rangle_{K}^{p-1}\right) \lesssim  C^p,
 $$ where the implicit constant is independent of the particular intervals $J,K$. 
\end{lem}

\begin{proof}
 By definition, $J, K$ have a unique (minimal) common dyadic ancestor interval $L$, and there exist nonnegative integers $s,t$ satisfying $s+t \leq N$ so that $\dyad(K,L)=s+1$ and $\dyad(J,L)=t+1$. Without loss of generality, we may suppose $K=L_{s -}^{m}$ and $J=L_{t -}^{n}.$ Again, for simplicity we only prove the result for $p=2.$ We consider test functions $f_1= w^{-1} \one_{L_{s-}^m}$ and $f_2= \one_{L_{t-}^n}.$ Straightforward computations yield
$$ \|f_1\|_{\Ltwo(w d\mu)}^2 = w^{-1}(L_s^{m-}), \quad \|f_2\|_{\Ltwo(w d \mu)}^2 =w(L_t^{n-}).$$
Choose the sequence $\alpha$ to be the unique sequence of signs satisfying $\alpha^I_{I_s^m, I_t^n} \langle f_1, h_{I_s^m} \rangle \langle h_{I_t^n}, w f_2 \rangle \geq 0$ for all $I \in \D$ (and $\alpha_{J,K}^{I}$ otherwise), and consider the operator $T\in\HS(s,t)$ defined by $\alpha$. Then we estimate, using the hypothesis :
\begin{align*}
C \|f_1\|_{\Ltwo(w d\mu)} \|f_2\|_{\Ltwo(w d\mu)} & \geq \langle |T f_1,w f_2 \rangle| \\
& = \left| \sum_{I \in \D} \alpha^I_{ I_s^m, I_t^n} \langle f_1, h_{I_s^m} \rangle \langle h_{I_t^n}, w f_2 \rangle\right|\\
& = \sum_{I \in \D}  |\langle f_1, h_{I_s^m} \rangle| |\langle h_{I_t^n}, w f_2 \rangle|\\
& \geq | \langle f_1, h_{L_s^m} \rangle| |\langle h_{L_t^n}, w f_2 \rangle|\\
& = \sqrt{m(L_s^m) m(L_t^n)} \left(\frac{w^{-1}(L_{s-}^{m})}{\mu(L_{s-}^{m})}\right) \left(\frac{w(L_{t-}^{n})}{\mu(L_{t-}^{n})}\right)\\
& \sim  \sqrt{\frac{m(J) m(K)}{\mu(J) \mu(K)}} \left(\langle w \rangle_{J} \langle w^{-1} \rangle_{K}\right)^{1/2} |f_1\|_{\Ltwo(w)} \|f_2 \|_{\Ltwo(w)}.
\end{align*} 
Squaring both sides and cancelling the $\Ltwo(w d\mu)$ norms then produces the desired estimate.
\end{proof}

The above result is the remaining bit that allows us to characterize the class of weights for which Haar shifts are $L^p$-bounded.

\begin{cor}\label{Weighted Characterization}
Let $p \in (1,\infty)$ and $N\in\N$. There exists a constant $C$ such that all operators 
$$
T \in \bigcup_{\substack{s,t \in \N: \\ s+t \leq N}} \HS(s,t)
$$
satisfy
$$ \|T\|_{\Lp(w d\mu) \rightarrow \Lp(w d\mu) } \leq C$$
if and only if $w \in A_p^{b}(\mu)$.     
\end{cor}

\begin{proof}
One direction immediately follows from Corollary \ref{WeightedBounds}. For the other,  we need to recall a relevant fact concerning dyadic operators of total complexity $0.$ In particular, given a sequence $\alpha=\{ \alpha_J\}_{J \in \D} \in \ell^\infty$, define the \emph{Haar multiplier} $\Pi_{\D}^\alpha$ as follows:
$$ \Pi_{\D}^\alpha(f):= \sum_{J \in \D} \alpha_J \langle f, h_J \rangle h_J$$
It is immediate that $\Pi_{\D}^\alpha$ is bounded on $\Ltwo(\mu).$ It is also well-known that for any Borel measure $\mu$, such an operator is bounded on $\Lp(\mu)$ for $1<p<\infty$. In fact, these operators admit sparse domination in the usual sense, see \cite{Lacey2017}. Moreover, the collection of such operators with $\|\alpha\|_{\ell^\infty} \leq 1$ is uniformly bounded on $\Lp(w d\mu)$ if and only if $w \in A_p(\mu)$, see \cite {TTV} for $p=2$. For $p \neq 2$, the sufficiency of the $A_p(\mu)$ condition may be obtained via the sparse domination, while the necessity may be obtained in the same way as $p=2$, see the proof of \cite{TTV}*{Proposition 2.3}. This fact together with Lemma \ref{PairBoundBelow} completes the proof.
\end{proof}
We end this subsection observing that Theorem \ref{th:thmB} follows from Proposition \ref{BadWeight} and Corollary \ref{Weighted Characterization}.

\subsection{Higher-Dimensional Haar Shifts} It is a natural question whether all results in this paper can be generalized to Haar shift operators defined on $\R^n$, $n>1$. A central point of departure from the one-dimensional case occurs in constructing the Haar system (say according to the abstract definition in \cite{LMP2014}), as there are more degrees of freedom. We make a couple of remarks concerning the higher dimensional case.

\begin{rem} The higher dimensional dyadic case for non-doubling measures was partially addressed in \cite{LMP2014}, where sufficient conditions were given for the weak $(1,1)$ bound of dyadic shifts defined with respect to two generalized Haar systems. In such a general context, we do not know if there are appropriate generalizations of the modified sparse forms and weight classes so that our sparse domination or weighted estimates remain true. It would be very interesting to investigate this question, as the higher dimensional theory could be significantly more complicated. The general martingale setting, with possibly nonatomic filtrations replacing the one generated by $\D$, is also open.
\end{rem}

\begin{rem}
In the special case that the Haar system $\{h_Q\}_{Q \in \D}$ consists of functions taking only two values, we expect there to be a parallel theory based on the condition
    $$
    \sup_{Q} \|h_Q\|_{\Lone(\mu)} \|h_{Q'}\|_{L^{\infty}(\mu)} \lesssim 1,
    $$
   for $Q, Q'$ with prescribed dyadic distance, which is the right higher dimensional analog of balanced. This case had already been identified in \cite{LMP2014}. Although we have not checked the details, the two-value Haar functions can be essentially regarded as one-dimensional objects, which is why we are confident this generalization will work without much modification to the arguments.     
\end{rem}

\end{document}